\documentclass[11pt]{article} 
\usepackage{fullpage}
\usepackage[T1]{fontenc} 
\usepackage[latin1]{inputenc}
\usepackage[english]{babel}
\usepackage{selinput}
\SelectInputMappings{adieresis={ä},germandbls={ß}}

\usepackage{amsmath} 
\usepackage{amssymb} 
\usepackage{amsthm} 
\usepackage{enumerate}
\newtheorem{theorem}{Theorem}
\newtheorem{lemma}{Lemma}

\theoremstyle{remark}
\newtheorem*{rmrk}{Remark}

\usepackage{graphicx} 
\usepackage{epstopdf}
\usepackage{subfig}
\usepackage[framed,numbered]{matlab-prettifier}
\usepackage{algorithm} 
\usepackage[noend]{algpseudocode}

\graphicspath{{./images/}}
\usepackage[numbers,sort]{natbib}
\usepackage{hyperref}
\usepackage{multirow}
\usepackage{fancyvrb}

\newcommand{\R}{\mathbb{R}}

\newcommand{\set}[1]{\left\lbrace #1 \right\rbrace}
\newcommand{\norm}[1]{\lVert #1 \rVert}
\newcommand{\abs}[1]{\left\lvert #1 \right\rvert}

\newcommand{\sprod}[2]{\langle #1, #2 \rangle}

\renewcommand{\b}[1]{\mathbf{#1}}
\renewcommand{\c}[1]{\mathcal{#1}}
\newcommand{\eps}{\varepsilon}
\newcommand{\diag}{\text{diag}}
\DeclareMathOperator*{\argmax}{arg\,max}
\DeclareMathOperator*{\argmin}{arg\,min}

\newcommand{\uargmin}[1]{\underset{#1}{\argmin}\;}
\newcommand{\uargmax}[1]{\underset{#1}{\argmax}\;}
\newcommand{\umin}[1]{\underset{#1}{\min}\;}
\newcommand{\umax}[1]{\underset{#1}{\max}\;}

\renewcommand{\abs}[1]{|#1|}

\usepackage{tikz}
\usetikzlibrary{shapes,arrows}

\newcommand{\rev}[1]{\textcolor{black}{ #1}}
\usepackage[numbers,sort]{natbib}
 
\title{Low-rank tensor approximations for solving multi-marginal optimal transport problems}
\usepackage{authblk}
\author[1]{Christoph Strössner}
\author[1]{Daniel Kressner}
\affil[1]{{\footnotesize \'Ecole Polytechnique F\'ed\'erale de Lausanne (EPFL), Institute of Mathematics, CH-1015 Lausanne, Switzerland}}
\date{}
\begin{document}
\maketitle
{\footnotesize \centering \vspace{-0.8cm}
\url{christoph.stroessner@epfl.ch}, \url{daniel.kressner@epfl.ch} \\ \vspace{1cm}
\centering {\normalsize \date{August 23, 2022}} \\ \vspace{1cm}
}

\begin{abstract} 
By adding entropic regularization, multi-marginal optimal transport problems can be transformed into tensor scaling problems, which can be solved numerically using the multi-marginal Sinkhorn algorithm. The main computational bottleneck of this algorithm is the repeated evaluation of marginals.  Recently, it has been suggested that this evaluation can be accelerated when the application features an underlying graphical model. In this work, we accelerate the computation further by combining the tensor network dual of the graphical model with additional low-rank approximations. We provide an example for the color transfer between several images,  in which these  additional low-rank approximations save more than $96\%$ of the computation time. 
\end{abstract}

\section{Introduction}

Classical optimal transport minimizes the transport cost between $m=2$ probability measures~\cite{Villani09,Benamou21}. 
For discrete measures, this problem can be expressed as the minimization of $\langle C,P \rangle$ for a non-negative cost matrix $C$. The
so{-}called transport plan $P$ is a non-negative matrix that has to satisfy marginal constraints, i.e., the column and row sums of $P$ are prescribed. 
The pioneering work of Cuturi~\cite{Cuturi13} established a relation between entropy regularized optimal transport and matrix scaling of $\exp(-C/\eta)$, where $\exp$ denotes the elementwise exponential and $\eta > 0$ denotes the regularization parameter. 
Scaling the rows and columns of $\exp(-C/\eta)$ such that marginal constraints are satisfied can be achieved numerically using the Sinkhorn algorithm~\cite{Sinkhorn64}, whose convergence speed can be accelerated using greedy coordinate descent~\cite{Altschuler17,Lin19c}, overrelaxation~\cite{Thibault17} or accelerated gradient descent~\cite{Dvurechensky18}. 
This matrix scaling approach allows one to solve much larger optimal transport problems compared to previous attempts based on solving the original linear program, which in turn has impacted various fields including image processing~\cite{Rabin14,Feydy17}, data science~\cite{Peyre19}, engineering~\cite{Mozaffari16} and machine learning~\cite{Kolouri17,Genevay18}. 

The classical optimal transport problem can be generalized to  a multi-marginal setting, in which the transport cost between $m\geq 3$ measures is minimized~\cite{Pass15}. 
Such multi-marginal problems arise in the areas of density functional theory~\cite{Marino17}, generalized incompressible flow~\cite{Benamou19}, neural networks~\cite{Cao19}, signal processing~\cite{Elvander20} and Wasserstein barycenters~\cite{Carlier15}. 
For discrete measures, the problem can be expressed as finding the non-negative transport plan tensor  $\c P$ of order $m${,} which minimizes $\langle \c C,  \c P \rangle$ subject to marginal constraints, where $\c C$ denotes a given non-negative cost tensor  of order $m$. In analogy to the the matrix case, after adding entropic regularization{,} the problem can equivalently be transformed into a tensor scaling problem for the tensor $\exp(-\c C / \eta)$~\cite{Benamou15}. 
The Sinkhorn algorithm can be generalized to solve this multi-marginal problem~\cite{Carlier21}; acceleration techniques via greedy coordinate descent are described in~\cite{Friedland20,Lin19}.

The multi-marginal Sinkhorn algorithm crucially relies on the repeated evaluation of marginals of the rescaled tensor $\exp(-\c C / \eta)$. 
The cost of computing such a marginal increases exponentially in $m$.
However, under certain assumptions on the structure of $\exp(-\c C / \eta)$, marginals can be computed much more efficiently.
For instance,  in certain applications  the structure of $\exp(-\c C/\eta)$  allows to specify the transport plan in terms of a graphical model.
When this graphical model does not contain circles, marginals can be computed efficiently using the belief propagation algorithm~\cite{Haasler21c,Fan21}.  
In particular, this includes tree structured cost tensors~\cite{Beier21,Haasler20}.
 When the model contains circles, the junction tree algorithm~\cite{Huang96} can be used to evaluate the marginals, but it might still incur a large computational cost.
 A different approach to possibly attain a complexity reduction is to replace $\exp(-\c C/\eta)$ by a low-rank approximation~\cite{Grasedyck13}, whose marginals can be evaluated efficiently.
For the classical case $m=2$, Altschuler et al.~\cite{Altschuler19} analyze the impact of the approximation error on the solution returned by the Sinkhorn algorithm.
For the case $m \geq 3$, asymptotic complexity bounds are derived in~\cite{Altschuler20b} for the specific case that $\c C$ has low tensor rank~\cite{Kiers00} and is  given explicitly in factored form.
Their results rely on a theoretical bound stating that $\exp(-\c C/\eta)$ is approximately low-rank in this situation.
The practical usefulness of these results is impeded by the fact that the elementwise exponential tends to increase (approximate) ranks drastically. For example, to obtain a reasonably good approximation of the elementwise exponential of a  random  $1000\times 1000$ rank-5 matrix by truncating all singular values smaller than $10^{-10}$ one easily ends up with a matrix of rank $800$ or larger.
Let us point out that low-rank approximations of $\exp(-\c C/\eta)$ should not be confused with low-rank approximations of the desired transport plan, as proposed in~\cite{Scetbon21}.

 {The main contribution of this work is to combine the ideas of exploiting underlying graphical models and using low-rank approximations to compute marginals more efficiently.
When the structure of transport plans is specified by a graphical model, we observe that the dual of this model is a tensor network~\cite{Robeva19}, which contains a tensor network representation of $\exp(-\c C/\eta)$.
At the same time low-rank approximations of $\exp(-\c C/\eta)$ can also be represented as tensor networks~\cite{Orus14}.
Facilitating this point of view, marginals can in both cases be computed by contracting~\cite{Ran20} the tensor network and the scaling parameters. 
In particular, the belief propagation and the junction tree algorithm for graphical models correspond to a particular order of contracting this network~\cite{Robeva19}.
For tensor networks derived from graphical models, we propose to potentially accelerate the computation of marginals further by replacing tensors in the network by low-rank approximations.
This yields a modified tensor network and can be seen as an approximation of $\exp(-\c C/\eta)$.} 
 {We provide theoretical bounds the error caused by introducing such approximations.
In Theorem~\ref{thm:LowRankError}, we provide a bound for the impact of using an approximation of $\exp(-\c C/\eta)$ on the entropically regularized transport cost.
This result is a generalization of the bound in~\cite{Altschuler19} for classical optimal transport problems.
In contrast to the asymptotic bound in~\cite{Altschuler20b}, our result contains explicit constants and does not assume that $\c C$ is low-rank. 
In Theorem~\ref{thm:EpsAccurateApprox}, we state how the parameters and tolerances need to be selected to obtain an accurate approximation of the original problem without regularization.
This generalizes previous results in~\cite{Lin19} from the case of using the tensor $\exp(-\c C/\eta)$ directly to our case of using an approximation instead.
In Lemma~\ref{lem:LowRankStructuredError}, we link approximations of parts in the tensor network to approximations of $\exp(-\c C/\eta)$.}
 {In our numerical experiments}, we provide an example illustrating that our approach  {to introduce low-rank approximations in the tensor network} is more efficient than directly working with the graphical model and more accurate than a direct tensor train approximation~\cite{Oseledets11} of $\exp(-\c C/\eta)$. 
We also demonstrate that our approach offers the potential to greatly speed up the computation of color transfer between several images without altering the resulting image significantly.

 {The remainder of this paper is structured as follows.
In Section~\ref{sec:Setting}, we define the multi-marginal optimal transport problem and summarize the convergence results for the multi-marginal Sinkhorn algorithm in~\cite{Lin19,Friedland20}.
The impact of approximating $\exp(-\c C/\eta)$ is analyzed in Section~\ref{sec:KernelApproximation}.
In Section~\ref{sec:StructuredTensors}, we derive the tensor network structure of $\exp(-\c C/\eta)$ for transport plans represented by graphical models.
Approximations of parts of this network are connected to approximations of $\exp(-\c C/\eta)$ in Section~\ref{sec:StructuredKernelApproximation}.
Our numerical experiments in Section~\ref{sec:NumerExperi} demonstrate how the multi-marginal Sinkhorn algorithm can be accelerated by introducing low-rank approximations into the tensor network representation of $\exp(-\c C/\eta)$.
\rev{Finally, Section~\ref{sec:Conclusion} concludes the paper.}
}

\subsection{Notation}
 
Throughout this work, we let $\norm{\cdot}_\infty$ denote the uniform norm and we let $\norm{\cdot}_1$ denote the $\ell^1$-norm. 
The Euclidean norm is denoted by $\norm{\cdot}_2$.
We let $\Delta^{n} = \{x \in \R^{n}_{>0}\colon \norm{x}_1=1 \}$ denote the set of strictly positive probability vectors of length $n$, where 
$\R_{>0}$ denotes the strictly positive real numbers. Further, we let $\R_+$ denote the non-negative real numbers.
For vectors $x,y \in \R^n$ we let $x \circ y$ denote the elementwise product.
The operations $\log$ and $\exp$ are always applied elementwise.
For a tensor $\mathcal{X} \in \mathbb{R}_+^{n_1\times\dots\times n_m}$ containing the joint probability distribution of $m$ random variables, we denote the $k$th marginal (distribution) by 
\begin{equation*}
    r_k(\c{X}) = \c{X} \times_1 \mathbf{1}_{n_1}^T \dots \times_{k-1} \mathbf{1}_{n_{k-1}}^T \times_{k+1} \mathbf{1}_{n_{k+1}}^T \dots \times_{m} \mathbf{1}_{n_m}^T, 
\end{equation*}
where $\mathbf{1}_{n} \in \R^{n}$ denotes  {the} vector of all ones  {and $\times_k$ denotes the mode-$k$ product~\cite{Kolda09}.}  {For a matrix $M \in \R^{\hat{n} \times n_{k}}$, the product is defined elementwise as 
\[ (\c X \times_k M)_{i_1,\dots,i_{k-1},j, i_{k+1}, \dots ,i_m} = \sum_{i_k=1}^{n_k} M_{ji_k} \c X_{i_1,\dots,i_m}\] for $1 \leq j \leq \hat{n}$, $1 \leq i_{\ell} \leq n_{\ell}$, $1\leq \ell \leq m$, $\ell \neq k$.
}
For tensors $\c X, \c Y \in \R^{n_1 \times \dots \times n_m}$, we let  $\sprod{\cdot}{\cdot}$ denote the 
standard inner product
\[\sprod{\c{X}}{\c{Y}} = \sum_{i_1 = 1}^{n_1} \dots \sum_{i_m = 1}^{n_m} \c{X}_{i_1,\dots,i_m} \c{Y}_{i_1,\dots,i_m}.\]

\section{Multi-marginal optimal transport and the Sinkhorn algorithm} \label{sec:Setting}

\subsection{Mathematical setting}

Given $m \geq 2$ marginals $r_k \in \Delta^{n_k},\ k=1,\dots,m$ and a cost tensor $\c{C} \in \mathbb{R}^{n_1 \times n_2 \times \dots \times n_m}_+$,
the discrete multi-marginal optimal transport problem is given by
\begin{equation}\label{eq:MOTproblem}
\min_{\c{P} \in B(r_1,\dots,r_m)}  \sprod{\c{P}}{\c{C}},
\end{equation}
where the set of feasible transport plans is given by \[B(r_1,\dots,r_m) = \set{\c{P} \in \R^{n_1\times \dots\times n_m}_{+} \rev{|} r_k(\c{P}) = r_k \text{ for } k = 1,\dots,m}.\] 
Note that~\eqref{eq:MOTproblem} is a linear optimization problem with $n_1\cdot n_2\cdots n_m$ degrees of freedom. 

To solve~\eqref{eq:MOTproblem} efficiently, it is common to add entropic regularization~\cite{Cuturi13}.
In the multi-marginal setting, the entropy takes the form \[H(\mathcal P) = -\sprod{\c P}{\log(\c P)}.\]
Given a regularization parameter $\eta > 0$, the regularized problem takes the form 
\begin{equation}
\label{eq:regularizeMOTproblem}
\min_{\mathcal P \in B(r_1,\dots,r_m)} V_\c{C}^{\eta}(\c{P}),
\end{equation} 
where \[V_\c{C}^{\eta}(\c{P}) := \sprod{\c{P}}{\c{C}}-\eta H(\c{P})\] is called entropic transport cost.
It is known that the regularized problem~\eqref{eq:regularizeMOTproblem} has a unique minimizer~\cite{Benamou15}, which takes the form
\begin{equation} \label{eq:tensorScalingForm}
    \c{P}_\eta^* = \c K \times_1 \diag(\exp(\beta_1))  \dots \times_m \diag(\exp(\beta_m)),
\end{equation}
where $\mathcal{K} := \exp(-\mathcal C/\eta)$ is called Gibbs kernel~\cite{Carlier21} and $\diag(\beta_k)$ denotes the diagonal matrix containing the entries of the so{-}called scaling parameters $\beta_k \in \R^{n_k}$ on its diagonal. The solution of the regularized problem~\eqref{eq:regularizeMOTproblem} converges to a solution of~\eqref{eq:MOTproblem} when $\eta \to 0$; \rev{see~\cite{Benamou16,Leonard12}}.

\subsection{Multi-marginal Sinkhorn algorithm} 

The multi-marginal Sinkhorn algorithm for solving~\eqref{eq:regularizeMOTproblem} proceeds by iteratively updating the scaling parameters in~\eqref{eq:tensorScalingForm}.
Let
\begin{equation}\label{eq:DefinitionPetat}
    \c{P}_\eta^{(t)} = \c K \times_1
\diag(\exp(\beta_1^{(t)})) \dots \times_m \diag(\exp(\beta_m^{(t)}))
\end{equation}
denote the scaled tensor obtained after $t$ iterations.
In each iteration one selects an index $k$ and updates the $k$th vector of scaling parameters via $\beta_k^{(t+1)} = \log(r_k) - \log(r_k(\c{P}_\eta^{(t)})) + \beta_k^{(t)}$, while the other vectors remain unchanged.
For the choice of indices it has been suggested to traverse them cyclically~\cite{Benamou15} or use a greedy heuristics~\cite{Friedland20,Lin19}. 
The multi-marginal Sinkhorn algorithm is terminated once the stopping criterion  
\begin{equation}\label{eq:stoppingCriterion}
\sum_{k=1}^m\norm{r_k(\c{P}_\eta^{(t)})- r_k}_1 \leq \eps_{\textsf{stop}}
\end{equation} is satisfied, where $\eps_{\textsf{stop}} > 0$ denotes a prescribed tolerance. Algorithm~\ref{alg:MMSinkhorn} summarizes the described procedure.

\begin{algorithm}[H]
\caption{Multi-marginal Sinkhorn algorithm}
\label{alg:MMSinkhorn}
\begin{algorithmic}[1]
\State \textbf{Input:} cost tensor $\mathcal C$, marginals $r_1, \dots, r_m$, regularization parameter $\eta$
\State \textbf{Output:} transport plan $\c{P}_\eta^{(t)}$
\State $\beta_i^{(0)} = \b{0} \in \mathbb{R}^{n_i}$ for $i=1,\dots,m$ and $\c K = \exp(-\c C/\eta)$
\For $t=0,1,\dots$ until stopping criterion~\eqref{eq:stoppingCriterion} is satisfied
\State $\c{P}_\eta^{(t)} = \c K \times_1 \diag(\exp(\beta_1^{(t)})) \dots \times_m \diag(\exp(\beta_m^{(t)}))$
\State Let $k_{\textsf{next}}$ denote the index of the scaling parameter that should be updated next.
\State $\beta_k^{(t+1)} = \left\{
\begin{array}{ll}
 \log(r_k) - \log(r_k(\c{P}_\eta^{(t)})) + \beta_k^{(t)} & \, k = k_{\textsf{next}} \\
 \beta_k^{(t)} & \, k \neq k_{\textsf{next}}
\end{array}
\right.$
\EndFor 
\end{algorithmic}
\end{algorithm} 

When using cyclic order, it follows from an interpretation as Bregman projections~\cite{Benamou15} that the multi-marginal Sinkhorn algorithm converges. For greedy strategies, Theorem~\ref{thm:SinkhornStops}  below summarizes the statements of~\cite[Theorem~4.3]{Lin19} and~\cite[Theorem~3.4]{Friedland20},  which bound the number of iterations until the stopping criterion is reached.
Note that the bound in a) gives a better rate with respect to $\eps_{\textsf{stop}}$, but the bound depends on $r_k$, whereas the bound in b) does not involve $r_k$. 

\begin{theorem}\label{thm:SinkhornStops}
Let $\c{C} \in \mathbb{R}^{n_1 \times \dots \times n_m}_+$, $r_k \in \Delta^{n_k}$ for $k=1,\dots,m$, $0 < \eta < \frac{1}{2}$, and $\eps_{\textsf{stop}} > 0$. 
\begin{enumerate}[a)]
    \item 
    Suppose that Algorithm~\ref{alg:MMSinkhorn}  selects the index of the scaling parameter in iteration $t$ according to
\[k_{\textsf{next}} = \argmax_{k \in \{1,\dots,m\}}  \mathbf{1}^T ( r_k(\c{P}_\eta^{(t)})-r_k) +   \mathbf{1}^T \big( r_k \circ \log \big({r_k(\c{P}_\eta^{(t)})} \circ r_k^{-1} \big)  \big),\] where
$r_k^{-1}$ denotes the elementwise inverse. Then the number of iterations to reach
the stopping criterion~\eqref{eq:stoppingCriterion}
is bounded by
\begin{equation}
    \label{eq:Thm1BoundA}
t \leq 2 + 2m^2 \eps^{-1}_{\textsf{stop}} \big( \eta^{-1} \norm{\c{C}}_\infty-\log \umin{1\leq k \leq m}\umin{1\leq i \leq n_k} (r_{k})_i \big).
\end{equation}
\item
Assume that $n = n_1 = \dots = n_m$ and 
suppose that Algorithm~\ref{alg:MMSinkhorn}  normalizes
$\c{P}_\eta^{(0)}$ to have $\ell^1$-norm $1$ and
selects the index of the scaling parameter in iteration $t$ according to 
\begin{equation}\label{eq:IndexSelectionFriedland}
k_{\textsf{next}} = \argmax_{k \in \{1,\dots,m\}} \Big\| r_k(\c{P}_\eta^{(t)}) - \frac{\sprod{r_k}{r_k(\c{P}_\eta^{(t)})}}{\norm{r_k}_2^2}r_k \Big\|_1.   \end{equation}
Then the number of iterations needed to reach the stopping criterion
\begin{equation}\label{eq:friedlandStoppingCrit}
\umax{k \in \{1,\dots,m\}} \Big\| r_k(\c{P}_\eta^{(t)}) - \frac{\sprod{r_k}{r_k(\c{P}_\eta^{(t)})}}{\norm{r_k}_2^2}r_k \Big\|_1 < \frac{\eps_{\textsf{stop}}}{2m}
\end{equation}
is bounded 
by
\[ 
t \leq  8 m^2(\sqrt{n}+1)^2 \eps_{\textsf{stop}}^{-2}
\log\big( \eta^{-1} \norm{\exp(-\c{C}/\eta)}_1 \big).
\]
When the alternative stopping criterion~\eqref{eq:friedlandStoppingCrit} is satisfied, the stopping criterion~\eqref{eq:stoppingCriterion} is also satisfied.
\end{enumerate}
\end{theorem}

A transport plan $\hat{\c{P}} \in B(r_1,\dots,r_m)$ is called $\eps$-approximate solution of the original problem~\eqref{eq:MOTproblem} if it satisfies
\[ 
\langle C, \hat{\c{P}}\rangle \leq  \umin{\c{P} \in B(r_1,\dots,r_m)} \sprod{\c{C}}{\c{P}} + \eps.
\]
The transport plan $\c{P}_\eta^{(t)}$ obtained from Algorithm~\ref{alg:MMSinkhorn} using either stopping criterion from Theorem~\ref{thm:SinkhornStops} is, in general, not in $B(r_1,\dots,r_m)$ because the marginal constraints $r_k = r_k(\c{P}_\eps^{(t)})$ are satisfied simultaneously only in the limit $t \to \infty$.  To fix this issue, rounding~\cite{Altschuler17} can be applied to enforce the marginal constraints on $P_\eps^{(t)}$  for finite $t$; see Algorithm~\ref{alg:Rounding}. 
Note that for a tensor of the form~\eqref{eq:DefinitionPetat}, the operation in line~\ref{algln:BetaModification} of Algorithm~\ref{alg:Rounding} can be phrased in terms of modifying $\beta^{(t)}_k$. Line~\ref{algln:RankOneUpdate} is a rank-$1$ update. 
In~\cite[Lemma 3.6]{Friedland20} and in~\cite[Theorem 4.4]{Lin19}, the following property of the resulting tensor is proven.
\begin{lemma}\label{lem:roundingproperties}
Let $\c{A} \in \mathbb{R}_{>0}^{n_1 \times \dots \times n_m}$ and $r_k \in \Delta^{n_k},\ k=1,\dots,m$. 
Let $\c{B}$ denote the output of Algorithm~\ref{alg:Rounding} applied to $\c{A}$ and $r_1,\dots,r_m$. 
Then $\c{B} \in B(r_1,\dots,r_m)$ and
\begin{equation*}
    \norm{\c{A}-\c{B}}_1 \leq 2 \sum_{k=1}^m \norm{r_k-r_k(\c{A})}_1.
\end{equation*}
\end{lemma}
\begin{algorithm}[H]
\caption{Rounding}
\label{alg:Rounding}
\begin{algorithmic}[1]
\State \textbf{Input:} tensor $\c{A} \in \R^{n_1 \times \dots \times n_m}_{>0}$, vectors $r_k \in \Lambda^{n_k}$ for $k =1,\dots,m$
\State \textbf{Output:} tensor $\c{B} \in \R^{n_1 \times \dots \times n_m}$  
\For $k=1,\dots,m$ 
\State $v = \min(r_k(\c A)^{-1} \circ r_k, \b{1}_{n_k})$, where the $\min$ is taken elementwise 
\State $\c{A} = \c{A} \times_k \diag(v)$ \label{algln:BetaModification}
\EndFor
\State $\c{B} = \c{A} + \norm{r_1-r_1(\c{A})}_1^{-(m-1)} \bigotimes_{k=1}^m (r_k-r_k(\c{A}))$, where $\bigotimes$ denotes the outer product, see~\cite{Kolda09} \label{algln:RankOneUpdate}
\end{algorithmic}
\end{algorithm}

Combining Algorithm~\ref{alg:Rounding} with Algorithm~\ref{alg:MMSinkhorn}, 
using the index selection~\eqref{eq:IndexSelectionFriedland} and the stopping criterion~\eqref{eq:stoppingCriterion}, it follows~\cite[Corollary 3.8]{Friedland20} that an $\eps$-approximate solution can be computed in \[\mathcal{O}(\eps^{-3}m^4n^{m+1}\log(n) (\max(\c{C})- \min(\c{C}))^3)\] operations, where $n_1 = \dots = n_K = n$. In practice, the computation can be accelerated by using slightly modified marginals and tolerances in the Sinkhorn algorithm~\cite{Lin19}, which ensure that the upper bound~\eqref{eq:Thm1BoundA} for $t$ is not dominated by small entries in $r_{k}$.
\section{Impact of approximating the Gibbs kernel} \label{sec:KernelApproximation}

To accelerate the computation of marginals in the multi-marginal Sinkhorn algorithm, we will replace the Gibbs kernel $\c K$ by an approximation $\tilde{\c K}$; thus replacing Algorithm~\ref{alg:MMSinkhorn} by  Algorithm~\ref{alg:ApproximateMMSinkhorn}. 
In this section, we will analyze the impact of this approximation on the  {transport cost of the computed} transport plan.
For $m = 2$, such an analysis can be found in~\cite[Theorem 5]{Altschuler19}.
 {In the following theorem, we generalize this result to the multi-marginal setting. 
Our proof closely follows the ideas in~\cite{Altschuler19}. 
In contrast to the asymptotic result in~\cite[Theorem~7.4]{Altschuler20b}, we provide explicit bounds.
}

\begin{algorithm}[H]
\caption{Multi-marginal Sinkhorn algorithm for a Gibbs kernel approximation}
\label{alg:ApproximateMMSinkhorn}
\begin{algorithmic}[1]
\State \textbf{Input:} approximation $\tilde{\c{K}} \in \R_{>0}^{n_1 \times n_2 \dots \times n_m}$ of Gibbs kernel $\c K = \exp(-\c C/\eta)$, marginals $r_1$, \dots, $r_m$
\State \textbf{Output:} transport plan $\tilde{\c{P}}^{(t)}$
\State $\beta_i^{(0)} = \b{0} \in \mathbb{R}^{n_i}$ for $i=1,\dots,m$
\For $t=0,1,\dots$ until stopping criterion~\eqref{eq:stoppingCriterion} is satisfied
\State $\tilde{\c{P}}^{(t)} = \tilde{\c{K}} \times_1 \diag(\exp(\beta_1^{(t)})) \dots \times_m \diag(\exp(\beta_m^{(t)}))$
\State Let $k_{\textsf{next}}$ denote the index of the scaling parameter  {that} should be updated next.
\State $\beta_k^{(t+1)} = \left\{
\begin{array}{ll}
 \log(r_k) - \log(r_k(\tilde{\c{P}}^{(t)})) + \beta_k^{(t)} & \ k = k_{\textsf{next}} \\
 \beta_k^{(t)} & \, k \neq k_{\textsf{next}} 
\end{array}
\right.$
\EndFor
\end{algorithmic}
\end{algorithm} 

\begin{theorem}\label{thm:LowRankError}
Let $\c{K} = \exp(- \c{C}/\eta)$ and assume that $\tilde{\c{K}} \in \R_{>0}^{n_1 \times n_2 \dots \times n_m}$ with $n_i\geq 2$ and $m\geq2$, satisfies  \[ \norm{\log(\c{K})-\log(\tilde{\c{K}})}_\infty \leq \eps_{\mathsf{log}} \leq 1.\]
Let $\tilde{\c{P}}$ denote the transport plan returned by  Algorithm~\ref{alg:ApproximateMMSinkhorn} with stopping criterion $\sum_{k=1}^m \norm{r_k(\tilde{\c{P}})-r_k}_1 \leq \eps_{\mathsf{stop}}$.
Then 
\begin{equation*} 
\abs{ V_\c{C}^{\eta}(\c{P}^*_\eta) - V_\c{C}^{\eta}(\tilde{\c{P}}) }\leq \eps_{V_\c{C}^{\eta}},
\end{equation*}
where $\c{P}^*_\eta = \rev{\argmin}_{\c{P} \in B(r_1,\dots,r_m)} V_\c{C}^\eta(\c{P})$ and  

\begin{align}
\eps_{V_\c{C}^\eta} = & \eta \big( \eps_{\textsf{log}} \big(2+\log\big(\frac{2}{\eps_{\textsf{log}}}\big)\big) + \frac{\eps_{\textsf{log}}}{2}\log\big(\big(\prod_{k=1}^m n_k\big) -1\big) + 2 \eps_{\mathsf{stop}} \log\big(\frac{1}{\eps_{\mathsf{stop}}} \big(\big(\prod_{k=1}^m n_k\big)-1\big)\big)\big) \nonumber \\ & +
   (\eps_{\mathsf{log}} + 2\eps_{\mathsf{stop}}) \norm{\c C}_\infty. \label{eq:MTepsBound}
\end{align}
\end{theorem}
\begin{proof}
We denote by $\Pi^S$ the operator mapping a given tensor $\c T \in \R^{n_1 \times \dots \times n_m}_{>0}$ to its unique~\cite{Franklin89} tensor scaling $\c U \in B(r_1,\dots,r_m)$ of the form $\c U = \c T \times_1 \diag(\gamma_1) \times_2 \dots \times_m \diag(\gamma_m)$ for some vectors $\gamma_k \in \R^{n_k}_{>0}$ for $1\leq k \leq m$. Observe that $\c{P}_\eta^* = \Pi^S(\c K)$.
Using the triangle inequality, we decompose the error into
\begin{align}
\abs{V_\c{C}^\eta(\c{P}_\eta^*)-V_\c{C}^\eta(\tilde{\c{P}})} \leq {} & \abs{V_\c{C}^\eta(\Pi^S(\c{K})) - V_\c{C}^\eta(\Pi^S(\tilde{\c{K}}))} \label{eq:Term1} \\
&+ \abs{V_\c{C}^\eta(\Pi^S(\tilde{\c{K}}))  - V_{\tilde{\c{C}}}^\eta(\Pi^S(\tilde{\c{K}}))} \label{eq:Term2} \\
&+ \abs{ V_{\tilde{\c{C}}}^\eta(\Pi^S(\tilde{\c{K}}))-  V_{\tilde{\c{C}}}^\eta(\tilde{\c{P}})} \label{eq:Term3} \\
&+ \abs{ V_{\tilde{\c{C}}}^\eta(\tilde{\c{P}}) - V_{\c C}^\eta(\tilde{\c{P}})} \label{eq:Term4}
\end{align}
where $\tilde{\c C} = - \eta \log( \tilde{ {\mathcal K}})$. 
We derive bounds for each of these terms; their combination yields inequality~\eqref{eq:MTepsBound}.

\paragraph{Bound for~\eqref{eq:Term1}:}
By definition of $V_\c{C}^\eta$, we have 
\begin{align*}
\abs{V_\c{C}^\eta(\Pi^S(\c{K})) - V_\c{C}^\eta(\Pi^S(\tilde{\c{K}}))} &\leq  \norm{\Pi^S(\c{K})-\Pi^S(\tilde{\c{K}})}_1 \norm{\c{C}}_\infty + \eta \abs{H(\Pi^S({\c{K}}))-H(\Pi^S(\tilde{\c{K}}))}.
\end{align*}
Because of $\uargmin{\c{P} \in B(r_1,\dots,r_m)} V_\c{C}^\eta(\c{P}) = \uargmin{\c{P} \in B(r_1,\dots,r_m)} \sprod{-\log(\c K)}{\c P}-H(\c P)$, it follows that
\begin{align*}
\norm{\Pi^S(\c{K})-\Pi^S(\tilde{\c{K}})}_1 &= 
\big\| {\uargmin{\c{P} \in B(r_1,\dots,r_m)} \sprod{-\log(\c{K})}{\c{P}} - H(\c{P})} - \big( {\uargmin{\tilde{\c{P}} \in B(r_1,\dots,r_m)} \langle -\log(\tilde{\c{K}})}, {\tilde{\c{P}} \rangle - H(\tilde{\c{P}})} \big) \big\|_1.
\end{align*}
Applying Lemma I in~\cite{Altschuler19} to the right hand side of this expression yields  
$\norm{\Pi^S(\c{K})-\Pi^S(\tilde{\c{K}})}_1 \leq \norm{\log{\c{K}}-\log{\tilde{\c{K}}}}_{\infty} \leq \eps_{\textsf{log}}$. 
From Theorem 6 in~\cite{Ho10} and Lemma D in~\cite{Altschuler19} we obtain \[\abs{H(\Pi^S({\c{K}}))-H(\Pi^S(\tilde{\c{K}}))} \leq  \eps_{\textsf{log}} \log\big(\frac{2}{\eps_{\textsf{log}}}\big) + \frac{\eps_{\textsf{log}}}{2}\log \big(\big(\prod_{k=1}^m n_k\big) -1\big).\] 

\paragraph{Bound for~\eqref{eq:Term2} and~\eqref{eq:Term4}:}
Using that $\|\Pi^S(\tilde{\c{K}})\|_1 = \|\tilde{\c P}\|_1 = 1$ we obtain 
\begin{align*}
\abs{V_\c{C}^\eta(\Pi^S(\tilde{\c{K}}))  - V_{\tilde{\c{C}}}^\eta(\Pi^S(\tilde{\c{K}}))}
&\leq \sprod{\c{C}}{\Pi^S(\tilde{\c{K}})}-\sprod{\tilde{\c{C}}}{\Pi^S(\tilde{\c{K}})}
\leq \norm{\c{C} - \tilde{\c{C}} }_\infty \leq \eta \eps_{\textsf{log}}, \\
\abs{ V_{\tilde{\c{C}}}^\eta(\tilde{\c{P}}) - V_\c{C}^\eta(\tilde{\c{P}})} &\leq  \sprod{\c{C}}{\tilde{\c P}}-\sprod{\tilde{\c{C}}}{\tilde{\c P}} \leq \norm{\c{C} - \tilde{\c{C}}}_\infty \leq \eta \eps_{\textsf{log}} .
\end{align*}

\paragraph{Bound for~\eqref{eq:Term3}:}
Using that the tensor $\tilde{\c P}$ is the unique minimizer of $ 
\uargmin{\c{P} \in B(r_1(\tilde{\c{P}}),\dots,r_m(\tilde{\c{P}}))} V_{\tilde{\c{C}}}^\eta(\c{P})$,
Lemma H in~\cite{Altschuler19} yields
\[
\abs{ V_{\tilde{\c{C}}}^\eta(\Pi^S(\tilde{\c{K}}))-  V_{\tilde{\c{C}}}^\eta(\tilde{\c{P}})} \leq \omega(d_H(B(r_1(\tilde{\c{P}}),\dots,r_m(\tilde{\c{P}})),B(r_1,\dots,r_m))),
\]
where $d_H(\cdot,\cdot)$ denotes the Hausdorff distance and \[\omega(x) = x \norm{\c{C}}_\infty + \eta \big(  x \log\big(\frac{2}{x} \big(\big(\prod_{k=1}^m n_k\big)-1\big)\big)\big).\] 
We can bound $d_H(B(r_1(\tilde{\c{P}}),\dots,r_m(\tilde{\c{P}})),B(r_1,\dots,r_m))$ by $2 \eps_{\textsf{stop}}$, since Algorithm~\ref{alg:Rounding} maps any $\c A \in B(r_1(\tilde{\c{P}}),\dots,r_m(\tilde{\c{P}}))$ to $\c B \in B(r_1,\dots,r_m)$ with $||\c A -\c B||_1 \leq 2 \eps_{\textsf{stop}}$ as stated in Lemma~\ref{lem:roundingproperties}.
This implies 
\[
\abs{ V_{\tilde{\c{C}}}^\eta(\Pi^S(\tilde{\c{K}}))-  V_{\tilde{\c{C}}}^\eta(\tilde{\c{P}})}
\leq 2\eps_{\textsf{stop}} \norm{\c{C}}_\infty + 2\eta \eps_{\textsf{stop}} \log\big(\frac{1}{\eps_{\textsf{stop}}} \big(\big(\prod_{k=1}^m n_k\big)-1\big)\big).
\] 
\end{proof}

The following theorem demonstrates how the previous result can be combined with Algorithm~\ref{alg:Rounding} to obtain $\eps$-accurate solutions.  {The proof is inspired by~\cite[Theorem 4.5]{Lin19} and~\cite[Theorem~3.7]{Friedland20}\rev{,} which state how $\eps$-accurate solutions can be computed using Algorithm~\ref{alg:MMSinkhorn}. 
Theorem~\ref{thm:EpsAccurateApprox} takes into account that we compute the transport plan using Algorithm~\ref{alg:ApproximateMMSinkhorn} based on a perturbed Gibbs kernel. 
}

\begin{theorem} \label{thm:EpsAccurateApprox}
 {L}et $\hat{\c{P}}$ be the tensor obtained by applying Algorithm~\ref{alg:Rounding} to $\tilde{\c{P}}$, where $\tilde{\c P}$ is obtained from Algorithm~\ref{alg:ApproximateMMSinkhorn} with $\tilde{\c K}$ fulfilling the assumptions of Theorem~\ref{thm:LowRankError}.
Let $ {\c P}^* = \argmin_{ {\c P} \in B(r_1,\dots,r_m)} \sprod{ {\c C}}{\c{P}}$.
Then it holds
\begin{equation*} 
    \sprod{\c{C}}{\hat{\c{P}}} - \sprod{{\c{C}}}{\c{P}^*} \leq \eps,
\end{equation*}
where $\eps =  {2 \eta \eps_{\textsf{log}}} + 2 \eta \sum_{k=1}^m \log(n_k)   +  {\rev{4} \norm{\c{C}}_\infty} \eps_{\textsf{stop}}$.
\end{theorem}
\begin{proof}
\rev{From $\c P^* \in B(r_1,\dots,r_m)$ and} $\c{P}_\eta^* = 
\argmin_{ {\c P} \in B(r_1,\dots,r_m)} \sprod{{{\c{C}}}}{\c{P}} - \eta H(\c{P})$ \rev{follows}
$\sprod{{{\c{C}}}}{\c{P}_\eta^*} - \eta H(\c{P}_\eta^*) \leq \sprod{{{\c{C}}}}{\c{P}^*} - \eta H(\c{P}^*)$. \rev{T}hus, 
\begin{align}
\sprod{{{{\c{C}}}}}{\c{P}_\eta^*} - \sprod{{{{\c{C}}}}}{\c{P}^*} \leq \eta H(\c{P}_\eta^*) - \eta H(\c{P}^*) \leq &\ \eta \sum_{k=1}^m \log(n_k), \label{eq:thm2boundetastarstar} 
\end{align}
where we use that $0 \leq H(\c{X}) \leq \sum_{k=1}^m \log(n_k)$ for any tensor $ { \mathcal{X} } \in \mathbb{R}_{\rev{+}}^{n_1\cdots n_m}$ \rev{with $||\c X||_1 = 1$}~\cite[ {Theorem~2.6.4}]{Cover06}.

\rev{Note that the marginals of $\tilde{\c P}$ are, in general, not equal to $r_1,\dots,r_m$. In order to compare $\tilde{\c P}$ and $\mathcal{P}_\eta^*$, we construct $\mathcal{Q} \in B(r_1(\tilde{\c P}), \dots, r_m(\tilde{\c P}))$ by applying Algorithm~\ref{alg:Rounding} to $\mathcal{P}_\eta^*$ with marginals $r_1(\tilde{\c P}), \dots, r_m(\tilde{\c P})$.
Since $\mathcal{P}_\eta^* \in B(r_1,\dots,r_m)$, Lemma~\ref{lem:roundingproperties} implies that $||\c Q - \mathcal{P}_\eta^*||_1 \leq 2 \eps_{\textsf{stop}}$. 
Hence,}
\begin{align}
    \label{eq:intermediateEQ}
\rev{
    \sprod{\c C}{{Q}} - \sprod{{{\c{C}}}}{\c{P}_\eta^*} \leq \norm{{{\c{C}}}}_\infty \norm{\c Q - \c{P}_\eta^*}_1 \leq 2 \norm{\c{C}}_\infty \eps_{\textsf{stop}}.
    }
\end{align}
The tensor $\tilde{\c P}$ is the unique scaling of $\tilde{\c K}$ with marginals $r_1(\tilde{\c P}), \dots, r_m(\tilde{\c P})$\rev{~\cite{Bapat82}}. It is thus the unique minimizer\rev{~\cite{Friedland20}} of 
$ 
\argmin_{ {\c P} \in B(r_1(\tilde{\c P}),\dots,r_m(\tilde{\c P}))} \sprod{{{\tilde{\c{C}}}}}{\c{P}} - \eta H(\c{P}),
$
where $\tilde{\c C} = -\eta \log( \tilde{\c K})$.
\rev{
Following the same argument as in~\eqref{eq:thm2boundetastarstar}, we obtain 
}
\begin{align}
\sprod{{{\tilde{\c{C}}}}}{\rev{\tilde{\c P}}} - \sprod{{{\tilde{\c{C}}}}}{\rev{\c Q}} \leq \eta H(\rev{\tilde{\c P}}) - \eta H(\rev{\c Q}) \leq &\ \eta \sum_{k=1}^m \log(n_k). \label{eq:THM3intermediate1} 
\end{align}
We further obtain
\begin{align}
 \sprod{{{\tilde{\c{C}}}}}{\rev{\c Q}} -  \sprod{{{{\c{C}}}}}{\rev{\c Q}} &\leq \norm{\tilde{\c C} - \c C}_\infty  \norm{\c Q}_1 = \norm{\tilde{\c C} - \c C}_\infty \leq \eta \eps_{\textsf{log}}, \label{eq:THM3intermediate2}\\
\sprod{{{{\c{C}}}}}{\tilde{\c{P}}}  - \sprod{{{\tilde{\c{C}}}}}{\tilde{\c{P}}} &\leq \norm{\tilde{\c C} - \c C}_\infty  \norm{\tilde{\c P}}_1 = \norm{\tilde{\c C} - \c C}_\infty \leq \eta \eps_{\textsf{log}}, \label{eq:THM3intermediate3}
\end{align}
where we use that $\norm{\tilde{\c P}}_1 = \norm{\rev{\c Q}}_1 = 1$.
Additionally, Lemma~\ref{lem:roundingproperties} yields 
\begin{equation}\label{eq:thm2boundhattilde}
\sprod{{{\c{C}}}}{\hat{\c{P}}} - \sprod{{{\c{C}}}}{\tilde{\c{P}}} \leq \norm{{{\c{C}}}}_\infty \norm{\hat{\c{P}}-\tilde{\c{P}}}_1 \leq 2 \norm{\c{C}}_\infty \eps_{\textsf{stop}}.
\end{equation}
By adding the inequalities~\eqref{eq:thm2boundetastarstar}--\eqref{eq:thm2boundhattilde} we obtain
\[
\sprod{\c{C}}{\hat{\c{P}}} - \sprod{{\c{C}}}{\c{P}^*} \leq  {2 \eta \eps_{\textsf{log}}} + 2 \eta \sum_{k=1}^m \log(n_k)   +  {\rev{4} \norm{\c{C}}_\infty} \eps_{\textsf{stop}}.
\]
\end{proof}

\begin{rmrk}
Note that for any  $\eps > 0$, we can find suitable $\eta, \eps_{\textsf{log}}, \eps_{\textsf{stop}}$ such that combining Algorithm~\ref{alg:ApproximateMMSinkhorn} and Algorithm~\ref{alg:Rounding} yields an $\eps$-accurate solution for the multi-marginal optimal transport problem~\eqref{eq:MOTproblem}.
 {
Analogously, we can find $\eta, \eps_{\textsf{log}}, \eps_{\textsf{stop}}$ such that~\eqref{eq:MTepsBound} is satisfied for any given $\eps_{V_{\c C}^\eta} > 0$.
In particular, we can first select $\eps_{\textsf{log}}$ and $\eps_{\textsf{stop}}$ based on $\norm{\c C}_{\infty}$. 
Afterwards, we can set $\eta$ sufficiently small.}
\end{rmrk}

\begin{rmrk}
 {
It might seem counter-intuitive that $\eps_{\textsf{log}}$ and $\eps_{\textsf{stop}}$ needs to be chosen inversely proportional to $\norm{\c C}_{\infty}$ in the previous remark.
This is caused by the chosen objective function in~\eqref{eq:regularizeMOTproblem}.
Let $\alpha = \norm{C}_\infty^{-1}$. 
Note that optimal solution of the regularized~\eqref{eq:regularizeMOTproblem} and the set of optimal solutions of the original optimal transport problem~\eqref{eq:MOTproblem} do not change when we replace both $\c C$ by $\alpha \c C$ and $\eta$ by $\alpha \eta$. 
This normalization changes the optimal value of the objective functions in~\eqref{eq:regularizeMOTproblem} and~\eqref{eq:MOTproblem} by $\alpha$, but it does not change the Gibbs kernel $\c K$. 
Thus, we can transform the error bounds in Theorem~\ref{thm:LowRankError} and~\ref{thm:EpsAccurateApprox} into bounds for this normalized problem, by multiplying $\eps_{V_{\c C}^\eta}$ and $\eps$ by $\alpha$.
To obtain a certain $\alpha \eps_{V_{\c C}^\eta}$ respectively $\alpha \eps$, we can select $\eps_{\textsf{log}}$ and $\eps_{\textsf{stop}}$ independently from $\norm{\c C}_\infty$ before determining an suitable $\eta$.
}
\end{rmrk}

\section{Tensor networks and graphical models}\label{sec:StructuredTensors} 

In applications, the cost tensor $\c C \in \R^{n_1 \times \dots \times n_m}_+$ usually carries additional structure. A broad class of structures  {leads to transport plans} defined via graphical models~\cite{Haasler21c}. In this case, the entries of $\c C$ take the form
\begin{equation}
    \label{eq:GraphicalModelCost}
    \c C_{I} = \sum_{\alpha \in F} \c C^{\alpha}_{I_{\alpha}} \quad \text{ for every } I = (i_1,\dots,i_m), \quad 1\leq i_k \leq n_k,\ 1\leq k \leq m,
\end{equation}
where the summation  {index tuples} $\alpha = (\alpha_1,\dots,\alpha_M)$ are  {contained in a fixed subset $F$ of \[\bigcup_{M=1}^m \{(\alpha_1,\dots,\alpha_M) \in \mathbb{N}^M| 1\leq \alpha_1 < \dots < \alpha_M \leq m\},\]} $I_{\alpha} := (i_{\alpha_1},\dots,i_{\alpha_M})$ and $\c C^{\alpha} \in \R_+^{n_{\alpha_1} \times \dots  \times n_{\alpha_M}}$.
The corresponding Gibbs kernel is given by
\begin{equation}
    \label{eq:GraphicalModelKernel}
    \c K_{I} = \prod_{\alpha \in F} \c K^{\alpha}_{I_{\alpha}} \quad \text{ for every } I = (i_1,\dots,i_m), \quad 1\leq i_k \leq n_k,\ 1\leq k \leq m,
\end{equation}
where $\c K^{\alpha} = \exp(-\c C^{\alpha}/\eta)$.
This Gibbs kernel can be represented in terms of a tensor network~\cite{Orus14}.
In general, a tensor network represents a high-order tensor \rev{that} is constructed by contracting several low-order tensors. 
By contraction we refer to the sum over a joint index in two low-order tensors.
A graph is used to describe precisely how the low-order tensors should to be contracted. 
Its vertices correspond to the low-order tensors.
Each edge corresponds to a contraction of the two low-order tensors corresponding to the vertices connected by the edge.

In the following, we describe how to construct the particular network for $\c K$.
First, we add each tensor $\c K^{\alpha}$ as a vertex.
For every $1\leq k \leq m$, we add an additional tensor $\c D^{(k)}$ as vertex.
For every $1\leq k \leq m$ and $\alpha \in F$, we add edges from $\c K^{\alpha}$ to $\c D^{(k)}$ if $k$ is contained in $\alpha$.
We then add one open edge to each $\c D^{(k)}$ \rev{that} corresponds to the index in the $k$th mode of $\c K$.
The order $d_k$ of the tensors $\c D^{(k)} \in \R^{n_k \times \dots \times n_k}$ is equal to the number of connected edges. 
Their entries are given by $\c D^{(k)}_{i_1,\dots,i_{d_k}} = \delta_{i_1i_2} {\delta_{i_2,i_3}}\cdots \delta_{i_{{d_k}-1}i_{d_k}}$, where $\delta$ denotes the Kronecker delta.
Each edge in the tensor network corresponds to the sum over the corresponding index in the connected vertices.
See Figure~\ref{fig:Duality} for an example. 

\begin{figure}[!ht]

\centering
\subfloat[Direct construction]{{\includegraphics[scale = 1.1]{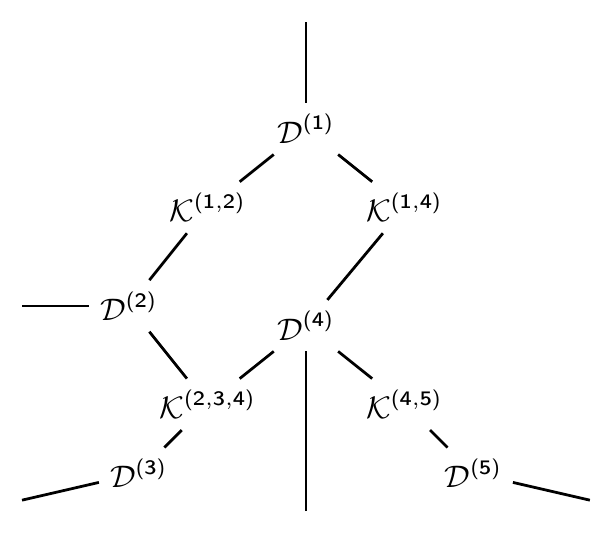}}}%
\qquad
\subfloat[After simplification]{{\includegraphics[scale = 1.1]{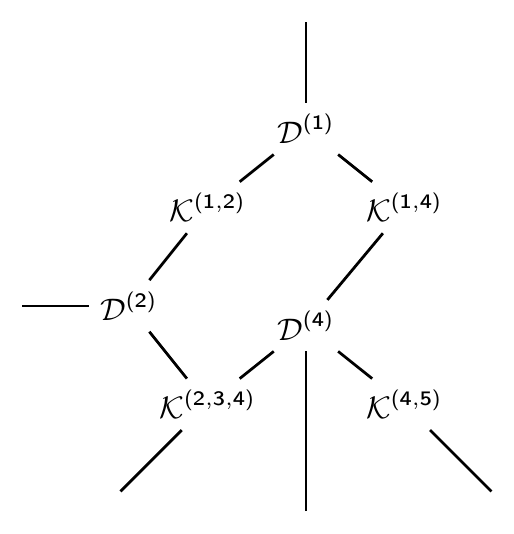}}}%
\captionsetup{singlelinecheck=off}
    \caption[.]{
 {Tensor network representation of $\mathcal{K}$ for a cost tensor of the form~\eqref{eq:GraphicalModelCost} constructed from $\c C^{( {1,2})}, \c C^{( {1},4)} \c C^{( {2,}3,4)}, \c C^{(4,5)}$. In~(b), we slightly simplified the network by contracting the identity matrices $\c D^{(3)}, \mathcal{D}^{(5)}$ with their connected open edges.
Summing over all internal edges yields the elementwise representation 
\begin{displaymath}
  \mathcal K_{i_1,i_2,i_3,i_4,i_5} = \sum_{j_1 = 1}^{n_1} \sum_{j_2 = 1}^{n_1} \sum_{j_3 = 1}^{n_2} \sum_{j_4 = 1}^{n_2} \sum_{j_5 = 1}^{n_4} \sum_{j_6 = 1}^{n_4} \sum_{j_7 = 1}^{n_4} 
 \mathcal D^{(1)}_{i_1,j_1,j_2} 
 \mathcal D^{(2)}_{i_2,j_3,j_4} 
 \mathcal D^{(4)}_{i_4,j_5,j_6,j_7} 
 \mathcal K^{(1,2)}_{j_1,j_3}
 \mathcal K^{(1,4)}_{j_2,j_5}
 \mathcal K^{(2,3,4)}_{j_4,i_3,j_6}
 \mathcal K^{(4,5)}_{j_7,i_5}.
\end{displaymath}
}
}
\label{fig:Duality}
\end{figure}

\begin{rmrk}
 {
We want to emphasize that this tensor network is closely related to the dual tensor network of the graphical model representing the transport plan~\cite{Robeva19}.
In the context of graphical models, the cost tensors $\c C$ is given in the form of~\eqref{eq:GraphicalModelCost}. The resulting transport plan $\c P^{(t)}_\eta$ defined in~\eqref{eq:tensorScalingForm} can be represented as tensor network by attaching the matrices $\diag(\exp(\beta_k))$ to the corresponding open modes of the tensor network representation of $\c K$.
At the same time, $\c P^{(t)}_\eta$ represents a discrete probability distribution\rev{,} which can be represented by a graphical model~\cite{Haasler21c}.
This graphical model and the tensor network of $\c P^{(t)}_\eta$ are duals of each other~\cite{Robeva19}.}
\end{rmrk}

 {
We can compute $\c K$ from the tensor network by contracting each of the internal edges sequentially. 
The contraction of one internal edge corresponds to the merging the connected vertices by evaluating of the sum over the index corresponding to the edge~\cite{Orus14}.
When the tensor network contains circles, multi-edges will occur, which can be contracted by summing over all the corresponding indices simultaneously.
The order of contracting the internal edges determines the degree of the occurring intermediate vertices and the computational complexity.
In our constructed tensor network, we can exploit the special structure of $\c D^{(k)}$ to contract all connected edges simultaneously.
The book~\cite{Ran20} discusses several heuristics to optimize the order of contractions.
Note that the optimal contraction sequence might still incur a large computational cost.
}

In Algorithm~\ref{alg:MMSinkhorn}, we need to evaluate the marginals of ${\c P_\eta^{(t)}}$ in each iteration. 
 {Let $\gamma_k^{(t)} = \exp(\beta_k^{(t)})$ for $1 \leq k \leq m$.}
Given a tensor network representation of $\c K$, the mode-$k$ marginals can be written as a contraction of the network after connecting the matrix $\diag( {\gamma^{(t)}_k})$ to the open edge corresponding to mode $k$, and the vectors $ {\gamma^{(t)}_{\tilde{k}}}$ for $\tilde{k}\neq k$ to their respective open edges. 
By contracting all inner edges of the resulting network, we obtain $r_k(\c P_\eta^{(t)})$. 
We refer to Figure~\ref{fig:StructuredMarginal} for examples. 
 {The depicted networks will again be used in the numerical experiments in Section~\ref{sec:NumerExperi}.}

\begin{figure}[!ht]
\centering
\subfloat[]{{\includegraphics[scale = 1.1]{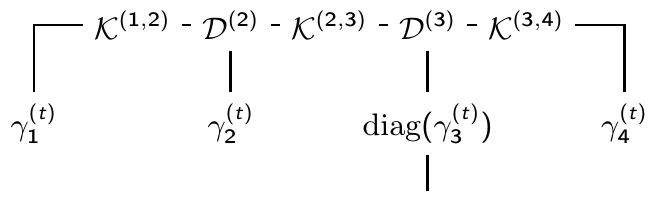}}}%
\qquad
\subfloat[]{{\includegraphics[scale = 1.1]{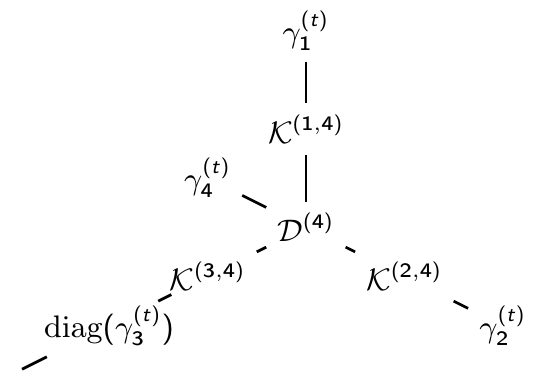}}}%
\caption{
Examples for the tensor network diagram representation of $r_3(\c{P}_\eta^{(t)})$ based on different  {tensor network structures for the Gibbs kernel}~\eqref{eq:GraphicalModelKernel}. 
}
\label{fig:StructuredMarginal}
\end{figure}

We give a brief example on how to efficiently contract the network depicted in Figure~\ref{fig:StructuredMarginal}(a). 
For the complexity analysis we assume $n_1=n_2=n_3=n_4=n$.
 {We first compute the vectors $v_j^{(1)} = \sum_{i=1}^n (\gamma_1^{(t)})_i \c K^{(1,2)}_{i,j}$ and $v_j^{(4)} = \sum_{i=1}^n \c K^{(3,4)}_{j,i} (\gamma_4^{(t)})_i $ for $1 \leq j \leq n$ in $n(2n-1)$ operations each,  {where operations refers to the required number of additions and multiplications.}
This corresponds to contracting the edge between $\gamma_1^{(t)}$ and $\c K^{(1,2)}$ as well as the edge between $\c K^{(3,4)}$ and $\gamma_4^{(t)}$.
In the next step, we contract the edges from $\c D^{(2)}$ to $v^{(1)},\gamma_2^{(t)}$ and $\c K^{(2,3)}$ simultaneously. 
This corresponds to computing the vector $v^{(2)}_j = \sum_{i=1}^n (v^{(1)} \circ \gamma_2^{(t)})_i \c K^{(2,3)}_{i,j}$
for $1 \leq j \leq n$ in $n+n(2n-1) = 2n^2$ operations, where we first compute the elementwise product before evaluating the matrix vector product.
Finally, we contract the remaining edges around $\c D^{(3)}$ to compute 
$r_3(\c{P}_\eta^{(t)}) = v^{(2)} \circ \gamma_3^{(t)} \circ v^{(4)}$ in $2n$ operations. 
In total, we need $6n^2$ operations to compute the marginal $r_3(\c{P}_\eta^{(t)})$.
\rev{Note that this contraction strategy corresponds to the following reordering of the sums in the definition of the marginal}
	\begin{align*}
		\rev{	r_3(P_\eta^{(t)})_{i_3}}
		&= \rev{\sum_{i_1=1}^n \sum_{i_2=1}^n \sum_{i_4=1}^n \mathcal K_{i_1, i_2}^{(1, 2)}\mathcal K_{i_2, i_3}^{(2, 3)}\mathcal K_{i_3, i_4}^{(3, 4)} (\gamma^{(t)}_1)_{i_1}(\gamma^{(t)}_2)_{i_2}(\gamma^{(t)}_3)_{i_3}(\gamma^{(t)}_4)_{i_4} }\\
		&= \rev{(\gamma^{(t)}_3)_{i_3}\cdot \Big( \sum_{i_2=1}^n\mathcal K_{i_2, i_3}^{(2, 3)}\cdot \Big( (\gamma^{(t)}_2)_{i_2} \cdot \Big( \sum_{i_1=1}^n\mathcal K_{i_1, i_2}^{(1, 2)} (\gamma^{(t)}_1)_{i_1} \Big) \Big)  \Big) \cdot\Big(\sum_{i_4=1}^n \mathcal K_{i_3, i_4}^{(3, 4)} (\gamma^{(t)}_4)_{i_4}\Big).		}
	\end{align*}
We can reuse the intermediate terms $v^{(1)},v^{(2)},v^{(4)}$ in the computation of the other marginals. 
This allows us to compute all four marginals in $12n^2 + 4n$ operations,
}
whereas {computing the marginals} based on the full tensor requires $\mathcal{O}(n^4)$ operations.

\begin{rmrk}
 {Instead of the tensor network based on an ordinary graph with special tensors $\c D^{(k)}$, we could consider a tensor network based on a hypergraph as in~\cite{Robeva19}. 
The structure of $\c D^{(k)}$ can be modeled by a single hyperedge containing all vertices connected to $\c D^{(k)}$.
The contraction of these hypergraph based networks is studied in~\cite{Huang21}.}
\end{rmrk}

\begin{rmrk}
 {Note that the structure of $\c K$ can also be exploited to compute marginals when applying Algorithm~\ref{alg:Rounding} to $\c P_\eta^{(t)}$. Storing the rank-$1$ update  in line~\ref{algln:RankOneUpdate}  separately in terms of its factors offers the potential to avoid storing the transport plan explicitly as a full tensor.}
\end{rmrk}

\section{Low-rank approximations in tensor networks} \label{sec:StructuredKernelApproximation}

Assuming that the Gibbs kernel is represented  {as in Equation}~\eqref{eq:GraphicalModelKernel}, we obtain a tensor network containing the coefficient tensors $\c K^{\alpha}$. 
The following lemma bounds the  {impact on the Gibbs kernel} when replacing each $\c K^{\alpha}$ by an approximation $\tilde{\c K}^{\alpha}$.

\begin{lemma}\label{lem:LowRankStructuredError}
 {Let $\c K$ be defined based on tensors $\c K^{\alpha} \in \R_{>0}^{n_{\alpha_1} \times \dots \times n_{\alpha_M}}$ for $\alpha \in F$ as in~\eqref{eq:GraphicalModelKernel}. 
Let $\tilde{\c K}^{\alpha} \in  \R_{>0}^{n_{\alpha_1}\times \dots \times n_{\alpha_M}}$ for $\alpha \in F$. We define $\tilde{\c K} \in \mathbb{R}^{n_1 \times \dots \times n_m}_{>0}$ elementwise as
\begin{equation}
\tilde{\c{K}}_{I} = \prod_{\alpha \in F} \tilde{\c K}^{\alpha}_{I_\alpha} \quad \text{ for every } I = (i_1,\dots,i_m), \quad 1\leq i_k \leq n_k,\ 1\leq k \leq m. \label{eq:lowRankStructuredK}
\end{equation}
}

Then
\[
  \norm{\log(\c{K}) - \log(\tilde{\c{K}})}_\infty \leq   \sum_{\alpha \in F} \norm{\log( {\c K}^{\alpha})-\log(\tilde{ { \c K}}^{\alpha})}_\infty.
\]
\end{lemma}
\begin{proof}
\begin{align*}
    \norm{\log(\c{K}) - \log(\tilde{\c{K}})}_\infty & = \max_{i_1,\dots,i_m} |\sum_{\alpha \in F} \log(\c K^{\alpha}_{I_\alpha}) - \sum_{\alpha \in F} \log(\tilde{\c K}^{\alpha}_{I_\alpha})| \\
    & \leq \max_{i_1,\dots,i_m} \sum_{ \alpha \in F} |\log(\c K^{\alpha}_{I_\alpha})-\log(\tilde{\c K}^{\alpha}_{I_\alpha})| \\
    & \leq \sum_{\alpha \in F} \norm{\log(\c K^{\alpha})-\log(\tilde{\c K}^{\alpha})}_\infty  {.}
\end{align*}
\end{proof}

Combining Lemma~\ref{lem:LowRankStructuredError} and Theorem~\ref{thm:LowRankError} implies that Algorithm~\ref{alg:ApproximateMMSinkhorn} with $\tilde{\c{K}}$ defined as in~\eqref{eq:lowRankStructuredK} yields an accurate approximation of the optimal transport plan when $\norm{\log(\c K^{\alpha})-\log(\tilde{\c K}^{\alpha})}_\infty$ is sufficiently small for all $\alpha \in F$. This allows one to replace each $\c K^{\alpha}$ by a low-rank approximation $\tilde{\c K}^{\alpha}$, which in turn accelerates the computation of tensor network contractions. 
For the example at the end of Section~\ref{sec:StructuredTensors}, the number of operations reduces from $\mathcal{O}(n^2)$ to $\mathcal{O}(nr)$ when every $\c K^{\alpha}$ is approximated by a rank-$r$ matrix of the form $\c K^{\alpha} = U^{\alpha}(V^{\alpha})^T$ with $U^{\alpha},V^{\alpha} \in \R^{n \times r}$ as depicted in Figure~\ref{fig:SVDsTT}(a).

\section{Numerical  {e}xperiments} \label{sec:NumerExperi}

All numerical experiments in this section were performed in MATLAB R2018b on a Lenovo Thinkpad T480s with Intel Core i7-8650U CPU and 15.4 GiB RAM.  {In Algorithms~\ref{alg:MMSinkhorn} and~\ref{alg:ApproximateMMSinkhorn} we select the next index to be updated using index selection~\eqref{eq:IndexSelectionFriedland} and we stop using stopping criterion~\eqref{eq:friedlandStoppingCrit} with $\eps_{\textsf{stop}} = 10^{-4}$.}
The code to reproduce these results is available from \url{https://github.com/cstroessner/Optimal-Transport.git}. 

\subsection{Proof of concept}\label{sec:ProofOfConcept}
In the following, we study the impact of approximating the Gibbs kernel on the transport cost.
 {We define a multi-marginal optimal transport problem, whose cost tensor is of the form studied in~\cite[Section 5.2]{Elvander20}.}  
 {Let $n=420$. For $1\leq k \leq4$, we generate random point sets $X^{(k)} =\{x^{(k)}_1,\dots, x^{(k)}_n\}$ by sampling the points $x^{(k)}_i \in \R^2$ independently randomly from the uniform distribution on $[0,1]^2$ for $1\leq i \leq n$.} 
We define $n \times n$ matrices $C^{(1,2)},C^{(2,3)},C^{(3,4)}$ with entries 
\[\c C^{(1,2)}_{i,j} = \norm{x^{(1)}_i-x^{(2)}_j}_2^2, \ \c C^{(2,3)}_{j,k} = \norm{x^{(2)}_j-x^{(3)}_k}_2^2, \ \c C^{(3,4)}_{k,l} = \norm{x^{(3)}_k-x^{(4)}_l}_2^2 \quad \text{for } 1\leq i,j,k,l \leq n.\]
We construct the cost tensor 
\begin{equation*}
    \mathcal{C}_{i,j,k,l} = \c C^{(1,2)}_{i,j}+\c C^{(2,3)}_{j,k}+\c C^{(3,4)}_{k,l} \quad \text{for } 1\leq i,j,k,l \leq n.
\end{equation*}
Let $\c K^{\alpha} = \exp(-\c C^{\alpha})$ for $\alpha \in \{(1,2),(2,3),(3,4)\}$. 
The Gibbs kernel $\mathcal{K} = \exp(-\mathcal{C})$  {is represented by the tensor network} shown in Figure~\ref{fig:StructuredMarginal}{(a)}. 

Let $r \leq n$. 
We compare two different approximations of $\mathcal{K}$.
For the first one, we  replace the matrices $\c K^{\alpha}$ by their rank-$r$ best approximations $\tilde{\c K}^{\alpha}$  using truncated singular value decompositions (SVDs) and define 
\[(\mathcal{K}_{\textsf{SVDs}})_{i,j,k,l} = \tilde{\c K}^{(1,2)}_{i,j}\cdot \tilde{\c K}^{(2,3)}_{j,k}\cdot \tilde{\c K}^{(3,4)}_{k,l}\quad \text{for } 1\leq i,j,k,l \leq n.\]
For the second approximation $\mathcal{K}_{\textsf{TT}}$, we compute a tensor train (TT) approximation~\cite{Oseledets11} of $\mathcal{K}$ with ranks $(r,r,r,r)$ using the TT-DMRG-cross algorithm~\cite{Savostyanov11} ignoring the underlying  {graph structure}. 
The tensor network representation of $\mathcal{K}_{\textsf{SVDs}}$ and $\mathcal{K}_{\textsf{TT}}$ is depicted in Figure~\ref{fig:SVDsTT}.
All four marginals can be computed in $\c O(nr)$ operations for $\mathcal{K}_{\textsf{SVDs}}$ and in $\c O(nr^2)$ operations for $\mathcal{K}_{\textsf{TT}}$ by contracting the tensor networks. 
In contrast, exploiting the  {graph structure in $\c K$} without low-rank approximations requires $\c O(n^2)$ operations.

\begin{figure}[!ht]
\centering
\subfloat[Based on $\mathcal{K}_{\textsf{SVDs}}$ ]{{\includegraphics[scale = 1.1]{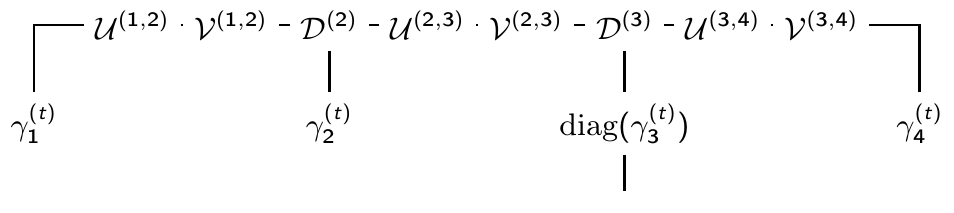}}}%
\qquad
\subfloat[Based on $\mathcal{K}_{\textsf{TT}}$]{{\includegraphics[scale = 1.1]{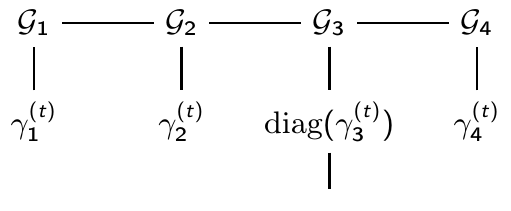}}}%
    \caption{Tensor networks for the computation of  $r_3(\c{P}_\eta^{(t)})$ for the example in Section~\ref{sec:ProofOfConcept}. 
    We express the truncated SVDs in $\mathcal{K}_{\textsf{SVDs}}$ as $\tilde{\c K}^{\alpha} = U^{\alpha}(V^{\alpha})^T$ with $U^{\alpha},V^{\alpha} \in R^{n \times r}$. 
    The TT cores in $\mathcal{K}_{\textsf{TT}}$ are denoted by $\c G_1 \in \R^{n \times r}$, $\c G_2, \c G_3 \in \R^{r \times n \times r}$ and $\c G_4 \in \R^{r \times n}$.}
    \label{fig:SVDsTT}
\end{figure}

Based on the tensors $\mathcal{K},\mathcal{K}_{\textsf{SVDs}},\mathcal{K}_{\textsf{TT}}$, we compute transport plans $\mathcal{P},\mathcal{P}_{\textsf{SVDs}},\mathcal{P}_{\textsf{TT}}$ by first applying Algorithm~\ref{alg:ApproximateMMSinkhorn} with marginals $r_k = \frac{1}{n_k}\cdot\mathbf{1}_{n_k}$ and then rounding the resulting tensor using Algorithm~\ref{alg:Rounding}. 
In Figure~\ref{fig:ProofOfConcept}, we compare the different transport plans. 
Note that we can efficiently evaluate the transport cost $\langle \c C,\c P \rangle$ using tensor network contractions of $\c C^\alpha$ and $\c P$ without evaluating the full tensors. 
We observe that the difference in transport cost of $\mathcal{P}_{\textsf{SVDs}},\mathcal{P}_{\textsf{TT}}$ and $\c P$ is much smaller than the norm of the difference of the logarithms of $\mathcal{K}_{\textsf{SVDs}},\mathcal{K}_{\textsf{TT}}$ and $\c K$.
The approximation $\c P_{\textsf{SVDs}}$  {that} exploits the  {graph structure} leads to slightly better approximations compared to $\mathcal{P}_{\textsf{TT}}$. 
We want to emphasize that computing $\c P_{\textsf{SVDs}}$ with $r=25$ is faster than using the  {graph structure} of $\c K$ directly and only leads to a difference in transport cost in the order of machine precision. 
Computing $\mathcal{P}_{\textsf{TT}}$ is faster than computing $\mathcal{P}$ for very small ranks.
 {The different scaling in the number of operations required to compute marginals leads to larger computation times for $\mathcal{P}_{\textsf{TT}}$ compared to $\mathcal{P}_{\textsf{SVDs}}$ for increasing values of $r$.}
 \rev{ We want emphasize that storing a tensor in $\R^{n \times n \times n \times n}$ explicitly would require more than $240$GB of memory. 
 Thus, it would not be feasible to solve this problem without exploiting either the underlying structure of $\mathcal{C}$ or the structure of the TT approximation.}

\begin{figure}[!ht]
\centering
\subfloat[Error analysis]{{\includegraphics[width=0.4\textwidth]{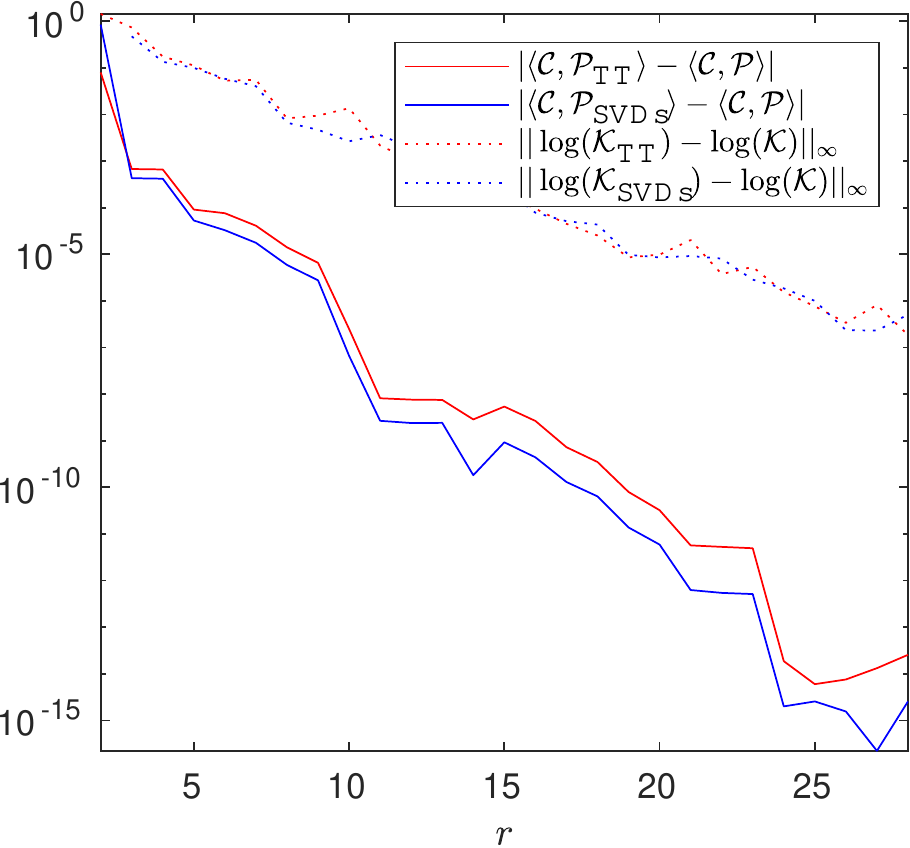}}}%
\qquad
\subfloat[Computation time]{{\includegraphics[width=0.4\textwidth]{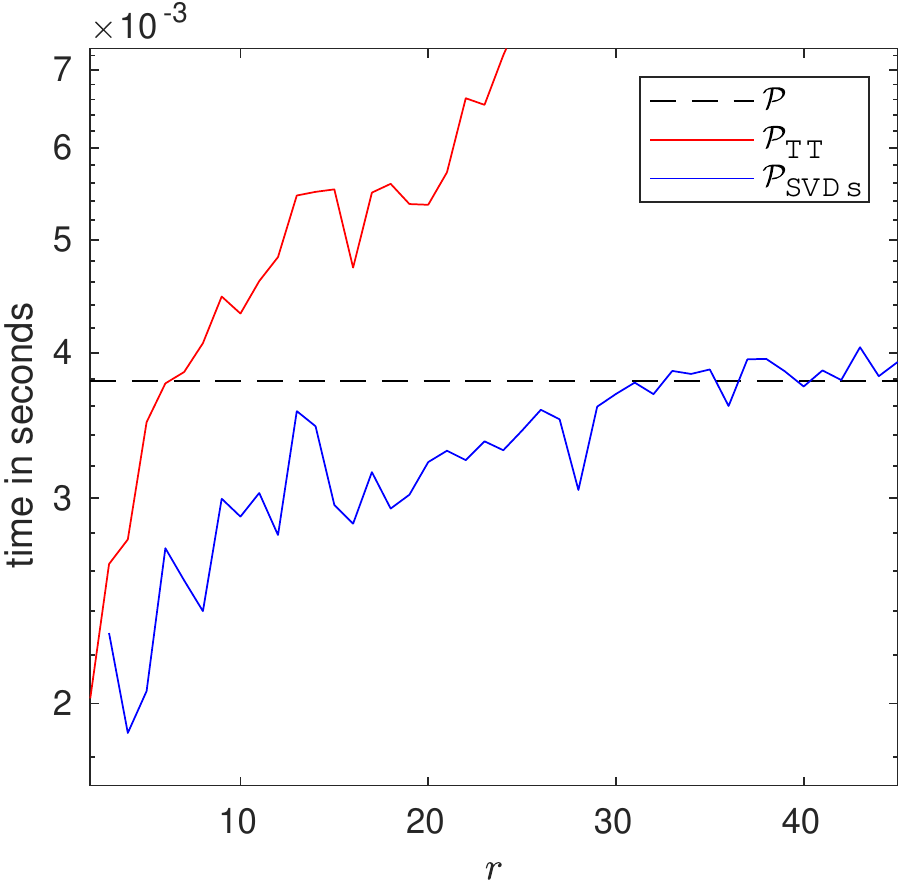}}}%
\caption{Difference in transport cost for the example in Section~\ref{sec:ProofOfConcept} for various ranks $r$. 
Left: We depict the difference in transport cost for $\mathcal{P}_{\textsf{SVDs}},\mathcal{P}_{\textsf{TT}}$ and $\mathcal{P}$ and an estimation of the norm of the difference of the logarithms of $\mathcal{K}_{\textsf{SVDs}},\mathcal{K}_{\textsf{TT}}$ and $\mathcal{K}$ based on $1\,000$ sample points. 
Right: Measured computation times for applying Algorithm~\ref{alg:ApproximateMMSinkhorn} and~\ref{alg:Rounding} to compute the transport plans exploiting the structures of $\mathcal{K},\mathcal{K}_{\textsf{SVDs}},\mathcal{K}_{\textsf{TT}}$.
 {Note that this time does not include the computation of $\mathcal{K}^{\alpha},\mathcal{K}_{\textsf{SVDs}},\mathcal{K}_{\textsf{TT}}$.}
}
\label{fig:ProofOfConcept}
\end{figure}

\begin{rmrk} {
Figure~\ref{fig:ProofOfConcept} shows that larger ranks lead to smaller approximation errors, but at the same time larger ranks increase the computation time. 
This needs to be balanced in practice. 
In particular, the rank needs to be sufficiently large such that the approximation $\tilde{\c K}$ of the Gibbs kernel is strictly positive, which implies that $\norm{\log(\c K) - \log(\tilde{\c K})}_{\infty}$ is bounded.
This can be achieved by choosing an approximation such that $\norm{\c K - \tilde{\c K}}_{\infty}$  is smaller than the smallest entry of $\c K$.
}\end{rmrk}

\begin{rmrk}
\rev{
The difference of the entropic cost for the tensors $\mathcal{P}_{\textsf{SVDs}},\mathcal{P}_{\textsf{TT}}$ and the optimal transport plan $\mathcal{P}_\eta^*$ is bounded by Theorem~\ref{thm:LowRankError}. 
We can assume that $\mathcal{P}$ is a good approximation of $\mathcal{P}_\eta^*$.
This would allow us to study the sharpness of the bound numerically.
However, the evaluation of the entropic cost requires the explicit computation of the full tensors, which is not feasible for $n=420$.
Instead, we repeat the experiment in Section~\ref{sec:ProofOfConcept} with a smaller $n$. 
The results are depicted in Figure~\ref{fig:sharpness}.
We find that the theoretical bound is much larger than the error observed in practice.
}
\end{rmrk}

\begin{figure}[!ht]
    \centering
    \includegraphics[width=0.4\textwidth]{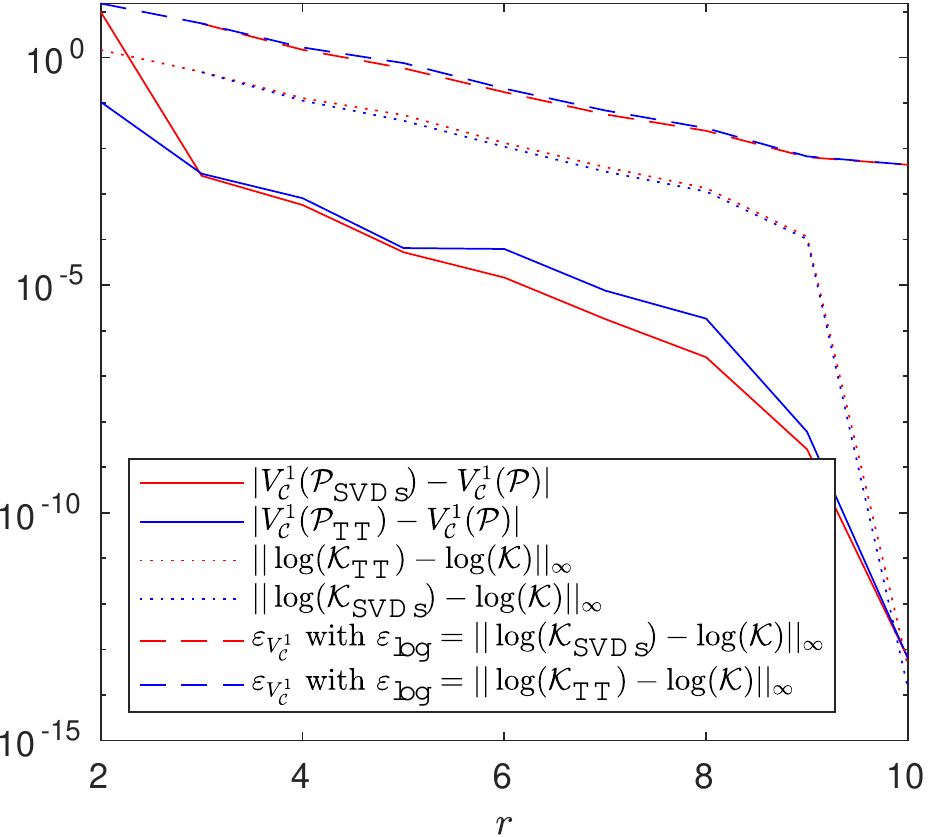}
    \caption{\rev{We repeat the experiment described in Section~\ref{sec:ProofOfConcept} with $n=10$ for various ranks $r$.
    We depict the difference in the entropic transport cost~\eqref{eq:regularizeMOTproblem}  of $\mathcal{P}_{\textsf{SVDs}},\mathcal{P}_{\textsf{TT}}$ and $\mathcal{P}$.
    Further, we depict the norm of the difference of the elementwise logarithm of $\mathcal{K}_{\textsf{SVDs}},\mathcal{K}_{\textsf{TT}}$ and $\mathcal{K}$.
    Based on these values we compute the value of $\eps_{V_{\mathcal C}^1}$  as defined in~\eqref{eq:MTepsBound}.}}
    \label{fig:sharpness}
\end{figure}

\subsection{Application: Color transfer  {from color \rev{b}arycenters}} \label{sec:Color}

In the following, we apply our algorithms for color transfer as in~\cite{Haasler21c}.   
We consider $k=4$ images containing $n = 100^2$ pixels each. 
Let $0 \leq \lambda_1, \lambda_2, \lambda_3$ such that $\lambda_1 + \lambda_2 + \lambda_3 = 1$. 
 {In a first step, we compute an approximation of the Wasserstein \rev{b}arycenter~\cite{Agueh11} with weights $\lambda = (\lambda_1,\lambda_2,\lambda_3)$ of the color of the first three images by solving the multi-marginal optimal transport problem \rev{in~\cite[Section~4.2]{Benamou15}}.
We then transfer the color from the approximation of the \rev{b}arycenter onto the fourth image by solving a two-marginal optimal transport problem~\cite{Rabin14}.
}

Let $x_i^{(k)} \in [0,1]^3$ denote the color of pixel $i$ in image $k$, where we treat the RGB values as element in $\R^3$ and rescale to $[0,1]^3$. 
 {Let $x^{(B)}_i = \lambda_1 x^{(1)}_i + \lambda_2 x^{(2)}_i + \lambda_3 x^{(3)}_i$ for $1\leq i \leq n$ \rev{as in~\cite{Benamou15}}.}
 We use these points to define a reference template~\cite{Wang13} for computing the approximation of the barycenter.
Let $\c C^{( {k,4})}_{i,j} = \norm{x^{(k)}_i-x^{( {B})}_j}_2^2$ for $1\leq i,j \leq n$ and $\c K^{( {k,4})} = \exp(-  \c C^{( {k,4})}/\eta)$ for $1 \leq  {k} \leq  {3}$.
We define the cost tensor 
\[
\mathcal{C}_{i,j,k,l} =  \lambda_1 \c C^{(1,4)}_{i,l}+ \lambda_2 \c C^{(2,4)}_{j,l}+ \lambda_3 \c C^{(3,4)}_{k,l} \quad \text{for } 1\leq i,j,k,l \leq n,
\]
 {and the Gibbs kernel tensor $\c K = \exp(-\rev{\c C}/\eta)$ for a given regularization parameter $\eta \geq 0$. Following the ideas in~\cite{Elvander20,Haasler21c}, we compute $r_B = r_4(\c{P}_\eta^*)$, where
\begin{equation} \label{eq:BarycenterScalingForm}
    \c{P}_\eta^* = \c K \times_1 \diag(\exp(\beta_1))  \times_2 \diag(\exp(\beta_2)) \times_3 \diag(\exp(\beta_3)) \times_4 \diag(\mathbf{1}_n),
\end{equation}
with scaling parameters $\beta_1,\beta_2,\beta_3 \in \R^{n}$ chosen such that $r_k(\c{P}_\eta^*) = \mathbf{1}$ for $1\leq k \leq 3$.
To compute an approximation $\c{P}_\eta$ of $\c{P}_\eta^*$ we run Algorithm~\ref{alg:MMSinkhorn} with cost tensor $\c C$ and marginals $r_1 = r_2 = r_3 = \mathbf{1}_n$ and index selection~\eqref{eq:IndexSelectionFriedland}, which we modify to use $\uargmax{k \in \{1,2,3\}}$ instead of $\uargmax{k \in \{1,2,3,4\}}$.
This modification ensures that we only update the first three scaling parameters, which leads to an approximation of the form~\eqref{eq:BarycenterScalingForm}~\cite[Theorem~3.5]{Haasler20}.
The Gibbs kernel corresponds to a star shaped graph structure as in Figure~\ref{fig:StructuredMarginal}{(b)}.
The approximation of the color \rev{b}arycenter is now given by the points $x^{(B)}$ with masses $r_B$.}

\begin{rmrk}
 {Proposition 3.4 in~\cite{Haasler20} states that multi-marginal optimal transport problems with star shaped graph structures can be decomposed into several independent two-marginal problems when all marginals are prescribed. 
This does not apply for the computation of~\eqref{eq:BarycenterScalingForm}, since $r_4(\c P_\eta^*)$ is unknown.}
\end{rmrk}

 {We now want to transfer the color from the approximation of the \rev{b}arycenter to the fourth image. 
We define the cost matrix $C_{i,j} = \norm{x^{(B)}_i-x^{( {4})}_j}_2^2$ for $1\leq i,j \leq n$ and Gibbs kernel matrix $ K= \exp(- C/\eta)$.
Let the matrix $P$ denote the approximate solution obtained from Algorithm~\ref{alg:MMSinkhorn} with cost matrix $C$ and marginals $r_1 = r_B$ and $r_2 = \mathbf{1}_n$. 
The color vector of the target image with transferred colors is now given by $x^*_j = \sum_{i=1}^n  P_{ij} x^{(B)}_i$ for $1\leq j \leq n$.
Note that we can transfer the color to several target images without recomputing the approximation of the \rev{b}arycenter.}

To accelerate the computation of marginals, we replace the  {the Gibbs kernel tensor $\c K$ and matrix $K$ by approximations $\tilde{\c K}$ and $\tilde{K}$ obtained by replacing ${\c K}^{(1,4)},{\c K}^{(2,4)},{\c K}^{(3,4)}$  {and $K$} by rank-$r$ approximations using the randomized SVD~\cite{Halko11}. 
Marginals  {and the target color vector $x^*$} can now be computed in $\mathcal{O}(nr)$ operations. 
}

\begin{figure}[!ht]
    \centering
\subfloat[Impact of $r$]{{\includegraphics[scale = 0.6]{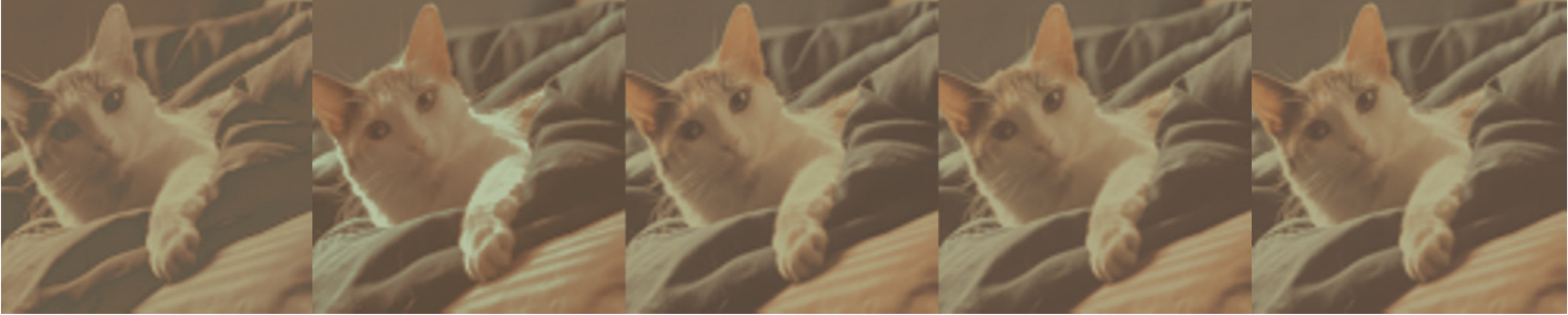}}}%
\qquad
\subfloat[Error decay]{{\includegraphics[scale = 1.1]{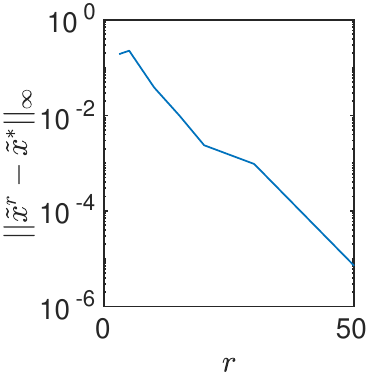}}}%
\qquad
\subfloat[Impact of $\lambda$]{{\includegraphics[scale = 0.6]{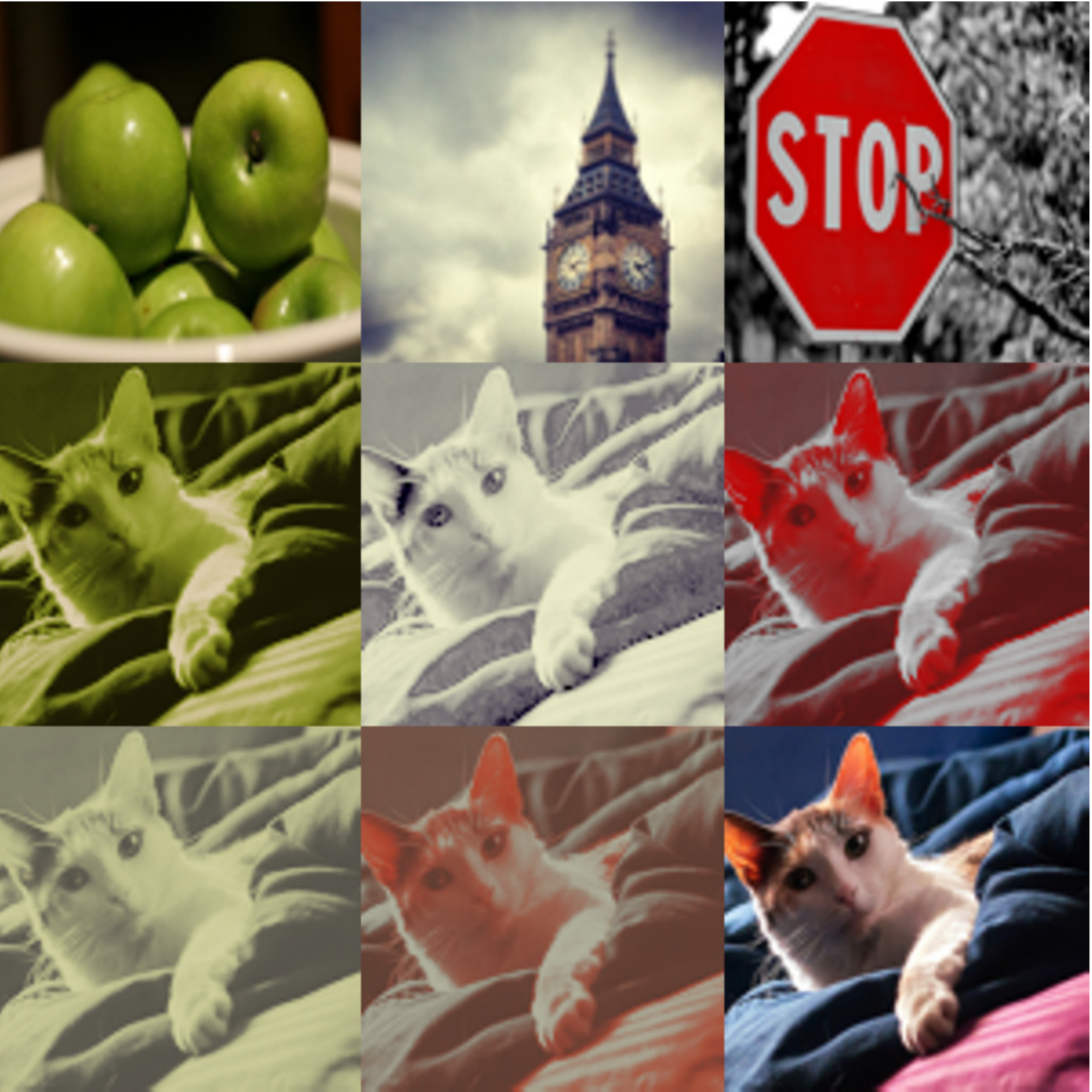}}}%
    \caption{
    Given the images displayed in the top row of (c), we use the method described in Section~\ref{sec:Color} to transfer their color to the bottom right image in (c). 
    For fixed $\lambda  = (1/3,1/3,1/3)$, we plot the resulting images for  {ranks} $r=3,r=5,r=10,r=50$ from left to right in (a). 
    The rightmost picture is obtained by using the full matrices $\c K^{\alpha}$  {and $K$} directly. 
    In (b), we plot $||\tilde{x}^{r} - \tilde{x}^{*}||_\infty$ where $\tilde x^{r}$ denotes the resulting image vector for a given  {rank} $r$ and $\tilde{x}^*$ is computed using the full matrices. 
    Moreover, we display the resulting image for different values of $\lambda$ in (c): middle row left to right $\lambda = (1,0,0),\ \lambda = (0,1,0),\ \lambda=(0,0,1)$, bottom row left $\lambda = (1/3,2/3,0)$, bottom row middle $\lambda = (1/5,1/5,3/5)$.}
    \label{fig:ColorTransfer}
\end{figure}

 {In the following numerical experiments, w}e set $\eta = 1/10$ and compute approximate transport plans $\tilde{\c P},  {\tilde{P}}$ by applying Algorithm~\ref{alg:ApproximateMMSinkhorn}  {to $\tilde{\c K}$ and $\tilde K$}. 
In Figure~\ref{fig:ColorTransfer}, we study the impact of $\lambda$ and $r$ onto the color transferred image for example images from the COCO data set~\cite{Coco14}. We observe that small values of $r$ suffice to accurately approximate the desired image. 
The computation including the assembling of the matrices and the randomized SVD takes $ {0.25}$ seconds for $r=50$, whereas using the  {full} matrices  {in} $\R^{10000\times 10000}$ directly in the tensor network takes $ {7.65}$ seconds, i.e. our low-rank approach reduces the computation time by over $96\%$.

\subsection{ {A tensor network with circles}}
 {
In the following, we describe an optimal transport problem arising in the context of Schrödinger bridges~\cite{Marino20,Haasler20}.
Let $m=5$ and $n = 40^2$. 
We denote by $x^{(i)}$ for $1 \leq i \leq n$ the $i$th grid point on the grid $\{1,\dots,40\}^2$.
Let ${\rev C}_{i,j} = ||x^{(i)} - x^{(j)}||_2^2$ for $1 \leq i,j \leq n$
and ${\rev K} = \exp(- {\rev C}/\eta)$.
We consider the Gibbs kernel 
\begin{equation}\label{eq:SchrödingerBridgeKernel}
    \mathcal{K}_{i_1,i_2,i_3,i_4,i_5} =  \prod_{\alpha \in F} {\rev K}_{I_\alpha} \quad \text{for } 1\leq i_1,i_2,i_3,i_4,i_5 \leq n,
\end{equation}
where $F = \{(1,2),(2,3),(3,4),(4,5)\}$.
Note that this corresponds to a tensor network structure similar to Figure~\ref{fig:StructuredMarginal}(a).
Given $r_1,r_5 \in \Delta^{n}$, we now consider the problem of finding scaling parameters $\beta_1,\beta_5 \in \R^n$ such that 
\begin{equation} \label{eq:SchrödingerBridgeScaling}
    \c{P}_\eta^* = \c K \times_1 \diag(\exp(\beta_1))  \times_2 \diag(\mathbf{1}) \times_3 \diag(\mathbf{1}) \times_4 \diag(\mathbf{1}) \times_5 \diag(\exp(\beta_5))
\end{equation}
satisfies $r_1(\c P_\eta^*) = r_1$ and $r_5(\c P_\eta^*) = r_5$.
In the context of Schrödinger bridges the marginals $r_k(\c P_\eta^*)$ describe how the initial distribution $r_1$ most likely evolved into $r_5$~\cite{Haasler20}. 
As in Section~\ref{sec:Color}, we can again compute approximations of $\c P_\eta^*$ by modifying Algorithm~\ref{alg:ApproximateMMSinkhorn} such that only  $\beta_1$ and $\beta_5$ are updated based on the prescribed marginals $r_1,r_5$\rev{; see~\cite[Section~III.A]{Haasler21c}}.
}

\begin{figure}[!ht]
    \centering
    \includegraphics[scale = 1.1]{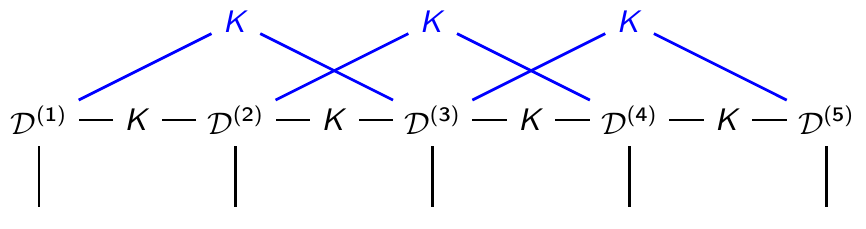}
    \caption{ {Tensor network diagram representation of $\mathcal{K}$ as defined in~\eqref{eq:SchrödingerBridgeKernel}. The black network is obtained for $F = \{(1,2),(2,3),(3,4),(4,5)\}$. The blue part depicts the additional nodes and vertices introduced by setting $F = \{(1,2),(1,3),(2,3),(2,4),(3,4),(3,5),(4,5)\}$.}}
    \label{fig:network15}
\end{figure}

 { 
The Schrödinger bridge problem is based on a Markov chain model, in the sense that each distribution only depends on the previous distribution.
We now introduce additional dependencies by setting $F = \{(1,2),(1,3),(2,3),(2,4),(3,4),(3,5),(4,5)\}$ in Equation~\eqref{eq:SchrödingerBridgeKernel}. 
\rev{The tensor network structure of $\mathcal{K}$ is depicted in Figure~\ref{fig:network15}.}
Note that this results in a graphical model for $\c P_\eta^*$ with window graph structure as in~\cite[Figure~2]{Altschuler20b}.
In Figure~\ref{fig:CircleContraction}, we depict the corresponding tensor network after replacing \rev{each matrix} ${\rev K}$ by a rank-$r$ approximation.
Marginals of this network can be computed in $\mathcal{O}(nr^4)$ operations. 
For instance, to compute $r_3(\c{P}_\eta^{(t)})$ we \rev{first} evaluate the colored tensors $\mathcal T^{[1]},\dots,\mathcal{T}^{[5]}$ in $\mathcal{O}(nr^4)$ operations by contracting the colored edges simultaneously using the structure of $\mathcal{D}^{(k)}$.
We then compute $\c T^{[1,2]} \in \R^{r \times r \times r}$ in $\mathcal{O}(r^4)$ operations by contracting $\mathcal T^{[1]}$ and $\mathcal T^{[2]}$ along their common edge. 
Analogously we compute $\c T^{[4,5]}$. 
We then contract $\c T^{[1,2]}$ and $\c T^{[4,5]}$ along their common edge in $\mathcal{O}(r^5)$ operations before contracting all edges of the resulting tensor simultaneously with $\c T^{[3]}$ in $\mathcal{O}(nr^4)$ operations.
\rev{We want to emphasize that contracting the network without replacing $K$ by low-rank approximations would not be feasible due to the required memory for storing the intermediate tensors.}
}

\begin{figure}[!ht]
    \centering
    \includegraphics[scale = 1.1]{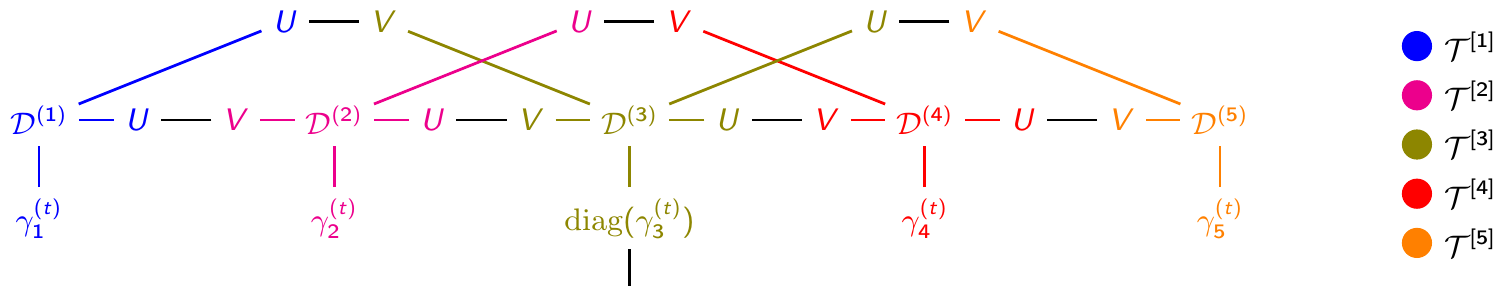}
    \caption{ {Tensor network diagram representation of $r_3(\c{P}_\eta^{(t)})$ corresponding to the graph structure~\eqref{eq:SchrödingerBridgeKernel} with  $F = \{(1,2),(1,3),(2,3),(2,4),(3,4),(3,5),(4,5)\}$. Here, we replace \rev{each} matrix $\rev{ K}$ by ${ U V^T}$ with ${ U, V}\in \R^{n\times r}$. The colors mark subtensors $\c T^{[1]},\dots\c T^{[5]}$ computed during the contraction.}}
    \label{fig:CircleContraction}
\end{figure}

 {
In Figure~\ref{fig:Schrödinger}, we study the impact of introducing these additional dependencies on the marginals $r_k(\c P)$, where $\c P$ denotes the computed approximation of~\eqref{eq:SchrödingerBridgeScaling}.
We observe that the additional dependencies lead to a more concentrated $r_3(\c P)$, in the sense that the mass is less spread out. 
Moreover, the additional dependencies lead to $r_2(\c P)$ and $r_4(\c P)$ being concentrated slightly closer to the center of the images.
}

\begin{figure}
    \centering
    \includegraphics[scale = 1.1]{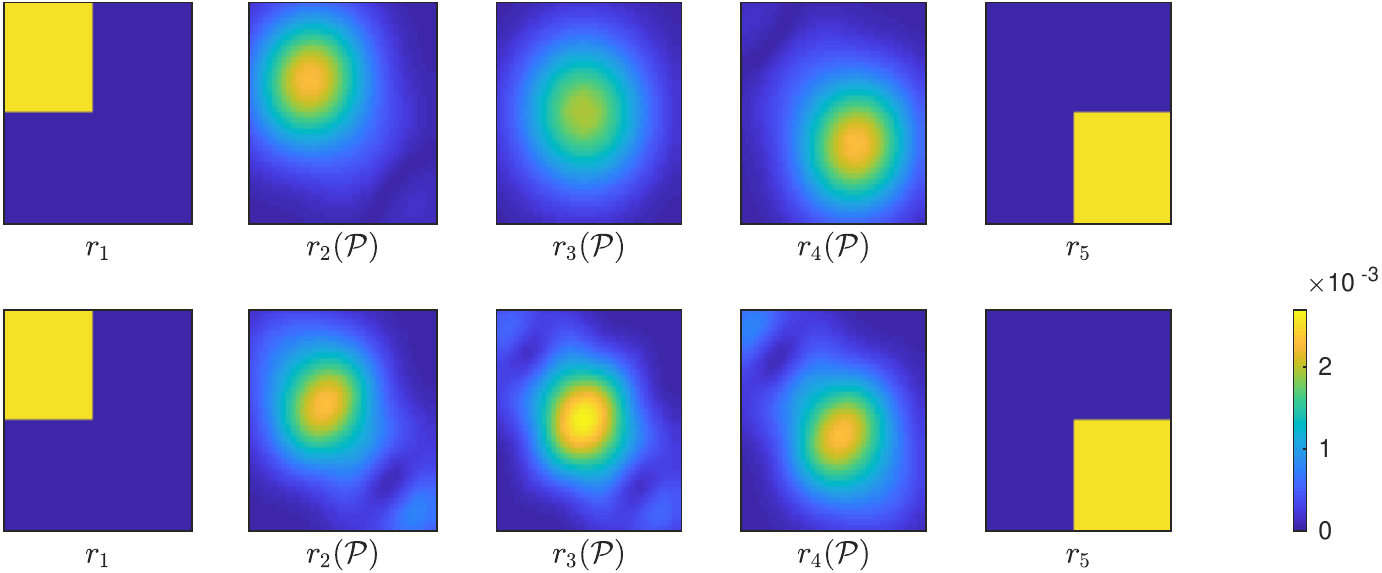}
    \caption{  { We consider the scaling problem~\eqref{eq:SchrödingerBridgeScaling} for $\eta = 0.1$. 
    We obtain an approximation $\tilde{\c K}$ of the Gibbs kernel by replacing $\hat{\c K}$ in~\eqref{eq:SchrödingerBridgeKernel} by a rank-$10$ approximation computed via the randomized SVD. 
    We apply Algorithm~\ref{alg:ApproximateMMSinkhorn} with index selection~\eqref{eq:IndexSelectionFriedland}, where we take $\argmax$ only over the set $\{1,5\}$. 
    Let $\c P$ denote the resulting transport plan.
    In each row we depict the prescribed $r_1,r_5$ as well as $r_2(\c P),r_3(\c P),r_4(\c P)$ reshaped into \rev{$\R^{40 \times 40}$}.
    The top row is obtained using $F = \{(1,2),(2,3),(3,4),(4,5)\}$ as in the Schrödinger bridge setting.
    The bottom row is obtained using $F = \{(1,2),(1,3),(2,3),(2,4),(3,4),(3,5),(4,5)\}$ with additional dependencies.} }
    \label{fig:Schrödinger}
\end{figure}

\begin{rmrk}
 {In this experiment, the graphical model of the transport plan contains circles. 
Hence, we can not apply the belief propagation algorithm directly~\cite{Haasler21c,Fan21}.
There exist generalizations such as the loopy belief propagation algorithm~\cite{Yedidia00}, which does not always converge, and the junction tree algorithm~\cite{Huang96}, which applies the belief propagation on a so-called junction tree.
This junction tree encodes the connection structure of the vertices.
It can be used to determine the contraction order for the corresponding tensor network representation of the Gibbs kernel~\cite[Section~4.2]{Robeva19}.
In particular, the minimal cost to contract the tensor network is bounded from above by the cost of the junction tree algorithm.}
\end{rmrk}
\section{Conclusion}\label{sec:Conclusion}
Multi-marginal optimal transport problems can be solved using the multi-marginal Sinkhorn algorithm, which suffers from the curse of dimensionality unless marginals can be computed efficiently.
In this paper, we analyze how approximations of the Gibbs kernel, which potentially drastically reduce the complexity of computing marginals, affect the solution of the transport problem. 
We demonstrate that the computation of marginals for transport plans defined via graphical models can be accelerated by introducing low-rank approximations in the tensor network representation of the Gibbs kernel.
We show that this approach can be faster and more accurate than direct low-rank approximations of the full Gibbs kernel. 
An application of our method is presented by the drastic reduction of the computation time for transferring colors from several images onto one target image.

In other applications, there are, however, several obstacles to put this into practice.
For instance, in density functional theory, the Coulomb cost leads to zero entries on the diagonal of the Gibbs kernel. 
Any approximation of the Gibbs kernel would need to approximate these entries exactly, which is not feasible without adding a sparse correction term to the low-rank approximations. 
Moreover, the underlying graphical model leads to a tensor network  that can not be contracted efficiently. 
In future work, the use of non-negative low-rank approximations could be of interest. 

\subsubsection*{Acknowledgements}
The authors would like to thank Virginie Ehrlacher for insightful discussions on this work.

\bibliographystyle{siam}

\end{document}